\documentclass[12pt]{amsart}

\pdfoutput=1

\usepackage{amsthm}
\usepackage{amsmath}
\usepackage{amstext}
\usepackage{amssymb}

\usepackage{caption}
\usepackage{verbatim} 
\usepackage{setspace}
\usepackage[dvips,letterpaper,margin= 1.3 in]{geometry}

\usepackage{graphicx}
\usepackage{hyperref}
\usepackage{pdflscape}
\usepackage{subcaption}

\usepackage{float}
\usepackage[utf8]{inputenc}
\usepackage[final]{pdfpages}

\newtheorem{theorem}{Theorem}[section]

\newtheorem{lemma}[theorem]{Lemma}
\newtheorem{remark}[theorem]{Remark}
\newtheorem{proposition}[theorem]{Proposition}

\newtheorem{definition}[theorem]{Definition}

\newtheorem{problem}[theorem]{Problem}

\newcommand{\dP}{dP$3$}

\title{Aztec Castles and the \protect\dP{}  Quiver}

\begin{document}

\author{Megan Leoni}
\address{Department of Mathematical Sciences, Durham University, Durham DH1 3LE, UK}
\email{megan.leoni@gmail.com}

\author{Gregg Musiker}
\address{School of Mathematics, University of Minnesota, Minneapolis, Minnesota 55455}
\email{musiker@math.umn.edu}

\author{Seth Neel}
\address{Department of Mathematics, Harvard University, Cambridge, MA}
\email{sethneel@college.harvard.edu}

\author{Paxton Turner}
\address{Department of Mathematics, Louisiana State University, Baton Rouge, Louisiana 70803}
\email{pturne7@tigers.lsu.edu}

\thanks{The authors were supported by NSF Grants DMS-\#1067183 and DMS-\#1148634.}


\date{July 18, 2014}


\maketitle

\begin{abstract}
Bipartite, periodic, planar graphs known as \textit{brane tilings} can be associated to a large class of quivers. This paper will explore new algebraic properties of the well-studied del Pezzo $3$ quiver and geometric properties of its corresponding brane tiling. In particular, a factorization formula for the cluster variables arising from a large class of mutation sequences (called $\tau-$mutation sequences) is proven; this factorization also gives a recursion on the cluster variables produced by such sequences. We can realize these sequences as walks in a triangular lattice using a correspondence between the generators of the affine symmetric group $\tilde{A_2}$ and the mutations which generate $\tau-$mutation sequences. Using this bijection, we obtain explicit formulae for the cluster that corresponds to a specific alcove in the lattice. With this lattice visualization in mind, we then express each cluster variable produced in a $\tau$-mutation sequence as the sum of weighted perfect matchings of a new family of subgraphs of the dP$3$ brane tiling, which we call \textit{Aztec castles}. Our main result generalizes previous work on a certain mutation sequence on the dP$3$ quiver in \cite{zhang}, and forms part of the emerging story in combinatorics and theoretical high energy physics relating cluster variables to subgraphs of the associated brane tiling. 
\end{abstract}

\tableofcontents

\begin{section}{Introduction}
\label{sec:introduction}
\par{First introduced and pioneered by Fomin and Zelevinsky in \cite{FZ} to study total positivity and dual canonical bases in semisimple Lie groups, cluster algebras have found a wealth of applications in many branches of mathematics including combinatorics, tropical geometry \cite{tropgeo}, Teichmuller theory \cite{teich}, and representation theory \cite{rep}. Cluster algebras are classes of commutative rings whose generators, known as cluster variables, are partitioned into subsets known as clusters. An iterative process called seed mutation provides a link between clusters and can be used to recover the complete set of generators given an initial seed \cite{zel}.

Meanwhile, theoretical physics has made intriguing advancements into the study of doubly-periodic, bipartite, planar graphs, known as brane tilings. These appear physically in string theory through the intersections of NS$5$ and D$5$-branes which are dual to a configuration of D3-branes probing the singularity of a toric Calabi-Yau threefold \cite{brane_dimer}. Brane tilings are inherently connected to the geometry of the threefold as well as the ($3+1$) dimensional supersymmetric gauge field theory that lives on the worldvolume of the D3-brane, which can be represented by a directed graph known as a quiver. There is significant interest in giving combinatorial interpretations for cluster variables arising from such quivers, see for example \cite{MS1,Musik, SchiffLee}.  Recent research has related a single mutation sequence on this quiver to a class of subgraphs of its brane tiling known as Aztec dragons \cite{zhang,CY}. This paper generalizes the work in \cite{zhang} by relating an infinite class of mutation sequences to a new, broader class of subgraphs of the brane tiling which we call \textbf{Aztec castles}}. 

\begin{subsection}{Quivers and Brane Tilings}
\label{subsec:quivbrane}
\par{A \textbf{quiver} $Q$ is a directed finite graph with a set of vertices $V$ and a set of edges $E$ connecting them whose direction is denoted by an arrow. For our purposes $Q$ may have multiple edges connecting two vertices but may not contain any loops or $2-$cycles. We can relate a cluster algebra with initial seed $\{x_{1},x_{2},\ldots,x_{n}\}$ to $Q$ by associating a cluster variable $x_{i}$ to every vertex labeled $i$ in $Q$ where $|V| = n$. The cluster is the union of the cluster variables at each vertex.

\definition{\textbf{Quiver Mutation:} Mutating at a vertex $i$ in $Q$ is denoted by $\mu_{i}$ and corresponds to the following actions on the quiver:
\begin{itemize}
\item For every 2-path through $i$ (e.g. $j \rightarrow i \rightarrow k$), add an edge from $j$ to $k$.
\item Reverse the directions of the arrows incident to $i$
\item Delete any 2-cycles created from the previous two steps. 
\end{itemize}}}
\noindent When we mutate at a vertex $i$, the cluster variable at this vertex is updated and all other cluster variables remain unchanged \cite{FZ}. The action of $\mu_{i}$ on the cluster leads to the following binomial exchange relation: 
\begin{equation*}
\label{eq: exchange relation}
x'_{i}x_{i} = \prod_{i \rightarrow j \; \mathrm{in} \; Q}x_{j}^{a_{i \rightarrow j}} + \prod_{j \rightarrow i \; \mathrm{in} \; Q}x_{j}^{b_{j \rightarrow i}}
\end{equation*}
where $x_i'$ is the new cluster variable at vertex $i$, $a_{i \rightarrow j}$ denotes the number of edges from $i$ to $j$, and $b_{j \rightarrow i}$ denotes the number of edges from $j$ to $i$.

Quivers describing gauge theories on the worldvolume of D3-branes probing toric Calabi-Yau singularities are also intimately connected to another planar graph known as a \textbf{brane tiling} (or dimer model). Unfolding the quiver onto the plane while maintaining the same edge connections leads to an infinite planar graph. The dual of this graph is the brane tiling: a bipartite, doubly-periodic, planar graph. 

\par{
Our main object of study is the quiver $Q$ associated to the third del Pezzo surface (\textbf{dP3}), illustrated in Figure \ref{fig:quiv_brane} with its associated brane tiling  \cite{brane_dimer, hanany_polygons}.  In particular, we consider the quiver associated with the Calabi-Yau threefold complex cone over the third del Pezzo surface of degree $6$ ($\mathbb{CP}^2$ blown up at three points), known as the del Pezzo $3$ quiver (\textbf{dP3}).  We focus on one of the four possible toric phases of this quiver.  In particular, $Q$ is Model I as in Figure 27 of \cite{franco_eager} (or Model II as first described in the physics literature in \cite{feng}). 
Note that later on, we will refer to a hexagon in the brane tiling consisting of the quadrilaterals (given in clockwise order) $5-3-1-4-2-0$ as a \textbf{6-cycle}. }

We can describe in more detail the unfolding process for this specific quiver in the manner of \cite{JMZ} (a description for generic quivers can be found where the procedure was first introduced in \cite{brane_dimer}). Denote by $E_{ij}$ the boundary edge from vertex $i$ to $j$ in our labeling of $Q$.

\begin{figure}[H]
    \centering
    \scalebox{0.6}{\includegraphics[keepaspectratio=true, width=\textwidth]{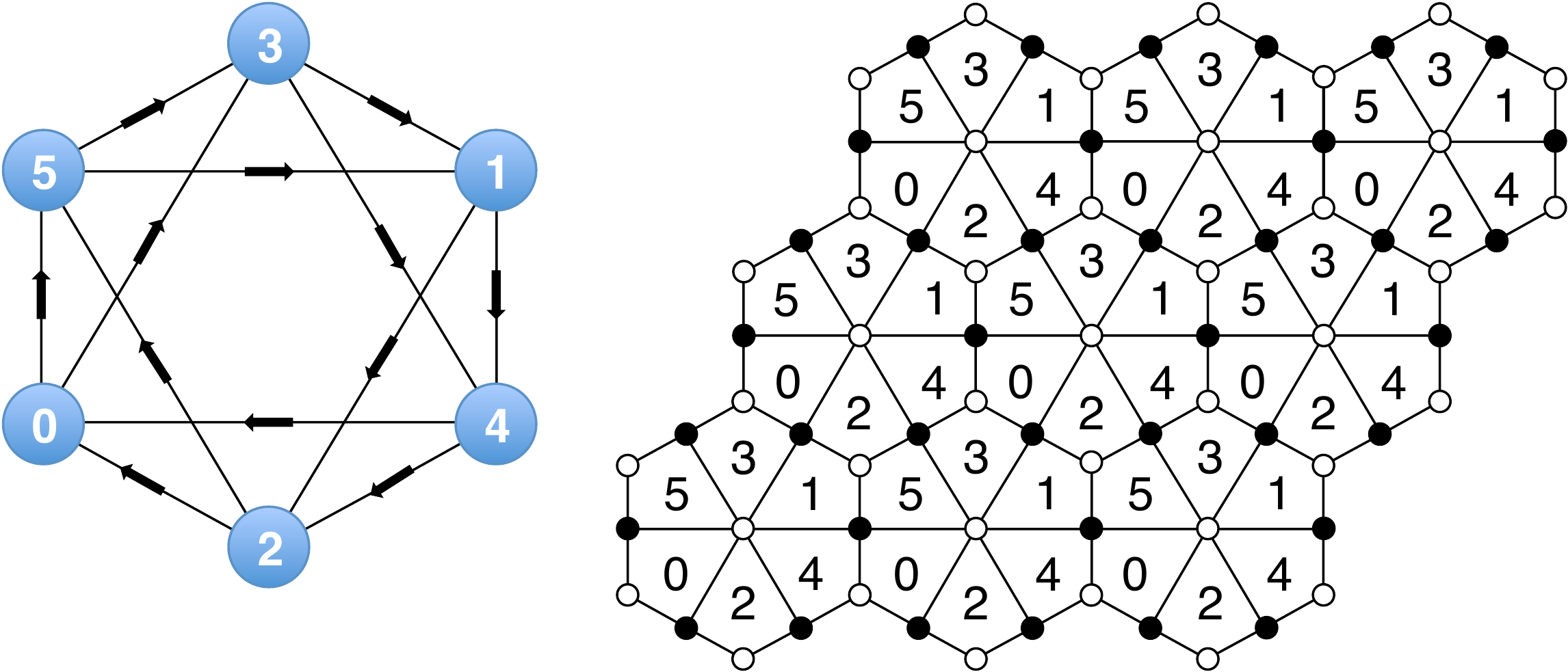}}
    \caption{\small \textbf{(Left)}: The dP3 quiver $Q$ and its associated brane tiling. \textbf{(Right)}: a doubly periodic, bipartite graph that is the planar dual to the unfolded quiver.}
    \label{fig:quiv_brane}
\end{figure}

 The dP$3$ brane tiling is associated to the following formal linear combination of closed cycles of the quiver, known as a \textbf{superpotential} (\textbf{$W$}) in a supersymmetric gauge theory where each edge in the unit cell of the brane tiling appears exactly twice, once for clockwise (positive) orientation and once for counter-clockwise (negative) orientation.
\begin{equation*}
\begin{split}
 W = E_{31}E_{14}E_{42}E_{20}E_{05}E_{53} + E_{34}E_{40}E_{03} + E_{12}E_{25}E_{51}\\
-E_{14}E_{40}E_{05}E_{51} - E_{34}E_{42}E_{25}E_{53} - E_{31}E_{12}E_{20}E_{03}
\end{split}
\end{equation*}
We can now unfold $Q$ into a planar digraph $\tilde{Q}$ composed of the given cycles such that the local configuration about vertex $i$ is the same in $\tilde{Q}$ and $Q$ , exhibited in Figure \ref{fig: unfolded}. The brane tiling illustrated in Figure \ref{fig:quiv_brane} is exactly the dual graph to $\tilde{Q}$, where bipartiteness is created by placing white vertices in positively oriented cycles and black vertices in negatively oriented cycles.

\begin{figure}[H]
    \centering
    \scalebox{0.9}{\includegraphics[keepaspectratio=true, width=85 mm]{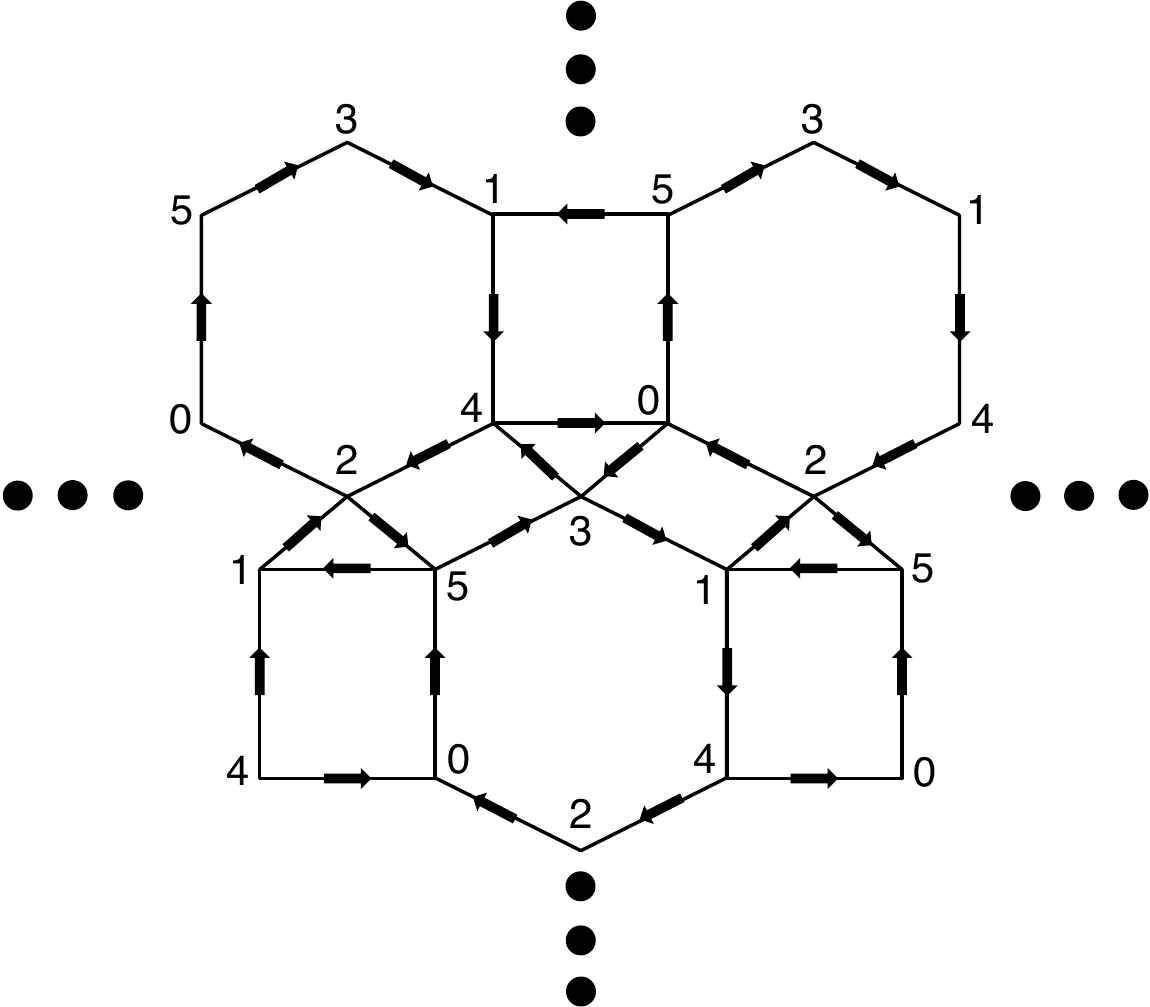}}
    \caption{\small A section of the unfolded dP3 quiver $\tilde{Q}$.}
    \label{fig: unfolded}
\end{figure}

\end{subsection}

\begin{subsection}{$\tau$-mutation Sequences}
\label{sec:tau}

Here we define a class of mutation sequences on $Q$ which we refer to as $\tau$-mutation sequences. The action of each $\tau$ mutates at antipodal vertices and has a symmetrical action on the quiver. This symmetry underlies the proofs of all of our main results.

\begin{definition}
\label{def:tau0}
Define the following pairs of mutations on $Q$.
\begin{itemize}
\centering
\item[] $\tau_{1}=\mu_{0} \circ \mu_{1}$
\item[] $\tau_{2}=\mu_{2} \circ \mu_{3}$
\item[] $\tau_{3}=\mu_{4} \circ \mu_{5}$
\end{itemize}

\end{definition}

Since antipodal vertices share no common edges, we observe that $\mu_{2i - 1}$ and $\mu_{2i - 2}$ commute for $i \in \{1,2,3\}$. 
Furthermore, the action of $\tau_{i}$ on the quiver exchanges the labels on vertices $2i-2$ and $2i-1$.  This motivates us to define 
the following actions on cluster seeds which are slight variants of the $\tau$-mutations.

\begin{definition}
\label{def:tau}
Define the following actions on $Q$.
\begin{itemize}
\centering
\item[] $\tau_{1}'=\mu_{0} \circ \mu_{1} \circ (01)$
\item[] $\tau_{2}'=\mu_{2} \circ \mu_{3} \circ (23)$
\item[] $\tau_{3}'=\mu_{4} \circ \mu_{5} \circ (45)$
\end{itemize}
where we apply a graph automorphism of $Q$ and permutation to the labeled seed after two mutations.

\end{definition}

One can then check that on the level of quivers and labeled seeds (i.e. ordered clusters), we have the following identities:
For all $i,j$ such that $1 \leq i \not = j \leq 3$
\begin{equation}
\label{eq: tau relations}
\begin{split}
 \tau_1'(Q) = \tau_2'(Q) = \tau_3'(Q) &= Q \\
(\tau_{i}')^{2} \{x_0,x_1,\dots, x_5\}&= \{x_0,x_1,\dots, x_5\} \\
(\tau_{i}'\tau_{j}')^{3} \{x_0,x_1,\dots, x_5\}&= \{x_0,x_1,\dots, x_5\}.
\end{split}
\end{equation}

A $\mathbf{\tau}$-\textbf{mutation sequence} is a mutation sequence of the form $\tau_{a_1} \tau_{a_2} \tau_{a_3} \ldots$. From now on we will abbreviate such a mutation sequence by its subscripts $a_1 a_2 a_3 \ldots$. The cluster variables produced after $\tau_{a_n} = \mu_{i} \circ \mu_{i+1}$ in such a sequence are denoted by $y_n$ and $y_n'$ where $y_n$ is the variable produced by $\mu_i$ and $y_n'$ is the variable produced by the latter mutation $\mu_{i+1}$ in $\tau_{a_n}$. By the symmetry of $Q$, $y_n$ and $y_n'$ are related by the action of the permutation $\sigma = (01)(23)(45)$ which flips antipodes, i.e. $\sigma y_n = y_n'$. It is also worth noting that $\sigma$ acts on the brane tiling as a $180^{\circ}$ rotation. 

\begin{remark} \label{rmk:taup} We will not use the $\tau_i'$'s elsewhere in this paper except in as far as to use these to identify $\tau$-mutation sequences and walks in the affine $\tilde{A}_2$ Coxeter lattice.  In particular, by using sequences of $\tau_i'$'s, we can identify labeled seeds, i.e. ordered clusters, with alcoves in the Coxeter lattice.  However since $\tau_i$'s and $\tau_i'$'s only differ up to simple permutations, we are still able to identify alcoves with unlabeled seeds, i.e. unordered clusters, which is all we need in this paper. 
\end{remark}

\end{subsection}

\begin{subsection}{Preliminary Definitions and Aztec Dragons}

We will adopt the weighting scheme on the brane tiling utilized in \cite{zhang}, \cite{speyer}, and \cite{GK} .  Associate the \textbf{weight} $\frac{1}{x_i x_j}$ to each edge bordering faces labeled $i$ and $j$ in the brane tiling. Let $M(G)$ denote the set of perfect matchings of a subgraph $G$ of the brane tiling. We define the weight $w(M)$ of a perfect matching $M$ in the usual manner as the product of the weights of the edges included in the matching under the weighting scheme. Then we define the weight of $G$ as
\[
w(G) = \sum_{M \in M(G)}w(M).
\]
We also define the \textbf{covering monomial}, $m(G)$, of an induced subgraph $G$ of the brane tiling as the monomial $x_0^{a_0}x_1^{a_1}x_2^{a_2}x_3^{a_3}x_4^{a_4}x_5^{a_5}$, where $a_j$ is the number of faces labeled $j$ enclosed in or bordering $G$\footnote{The covering monomial has a more general definition suitable for other contexts. See \cite{jeong} and \cite{JMZpaper}.}. Figure \ref{fig: cov mon} illustrates an example of the quadrilaterals included in the covering monomial of a small subgraph, outlined in red.

\begin{figure}[h]
    \centering
    \scalebox{0.75}{\includegraphics[keepaspectratio=true, width=50 mm]{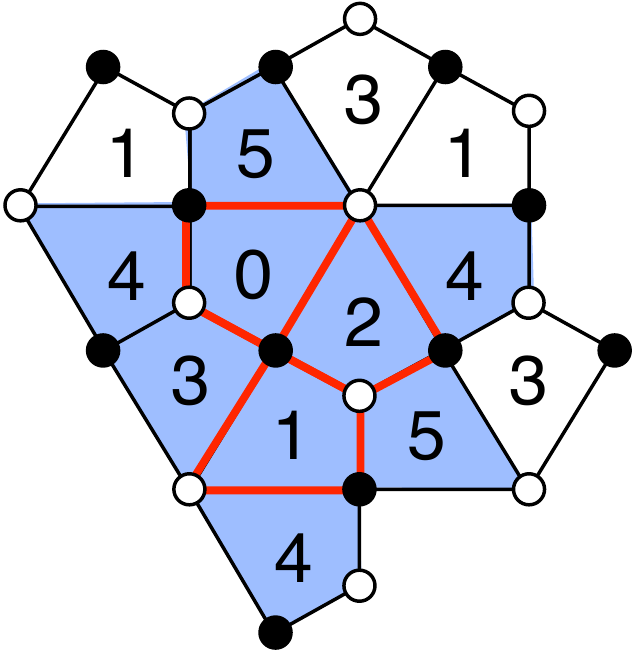}}
	\caption{\small The covering monomial of the subgraph outlined in red includes the blue quadrilaterals and is given by $x_{0}x_{1}x_{2}x_{3}x_{4}^{3}x_{5}^{2}$.} 
	\label{fig: cov mon}
\end{figure}

Finally, to make notation more concise in later proofs, it will be useful to define the product of the covering monomial and weight of a subgraph $G$ as 
\begin{equation*}
c(G) = w(G)m(G).
\end{equation*}

Zhang studied the mutation sequence $123123\ldots$ in \cite{zhang} and found that the cluster variable obtained after each mutation in the sequence can be expressed as the weighted perfect matchings of a family of subgraphs of the dP3 brane tiling. In fact these graphs turned out to be the family of \textbf{Aztec dragons} $\{D_{n/2}\}_{n \in \mathbb{N}}$ introduced and enumerated in \cite{CY}. The first few Aztec dragons are exhibited in Figure \ref{fig: zhang diamonds}. Specifically in the notation described above, Zhang proved that $y_n = c(D_{n/2})$ (respectively, $y_n' = c(\sigma D_{n/2})$). Zhang's findings have provided the starting point for our main contribution: relating the weighted perfect matchings of subgraphs of the dP3 brane tiling to cluster variables for all possible $\tau$-mutation sequences.

\begin{figure}[H]
    \centering
    \scalebox{0.85}{\includegraphics[keepaspectratio=true, width=120 mm]{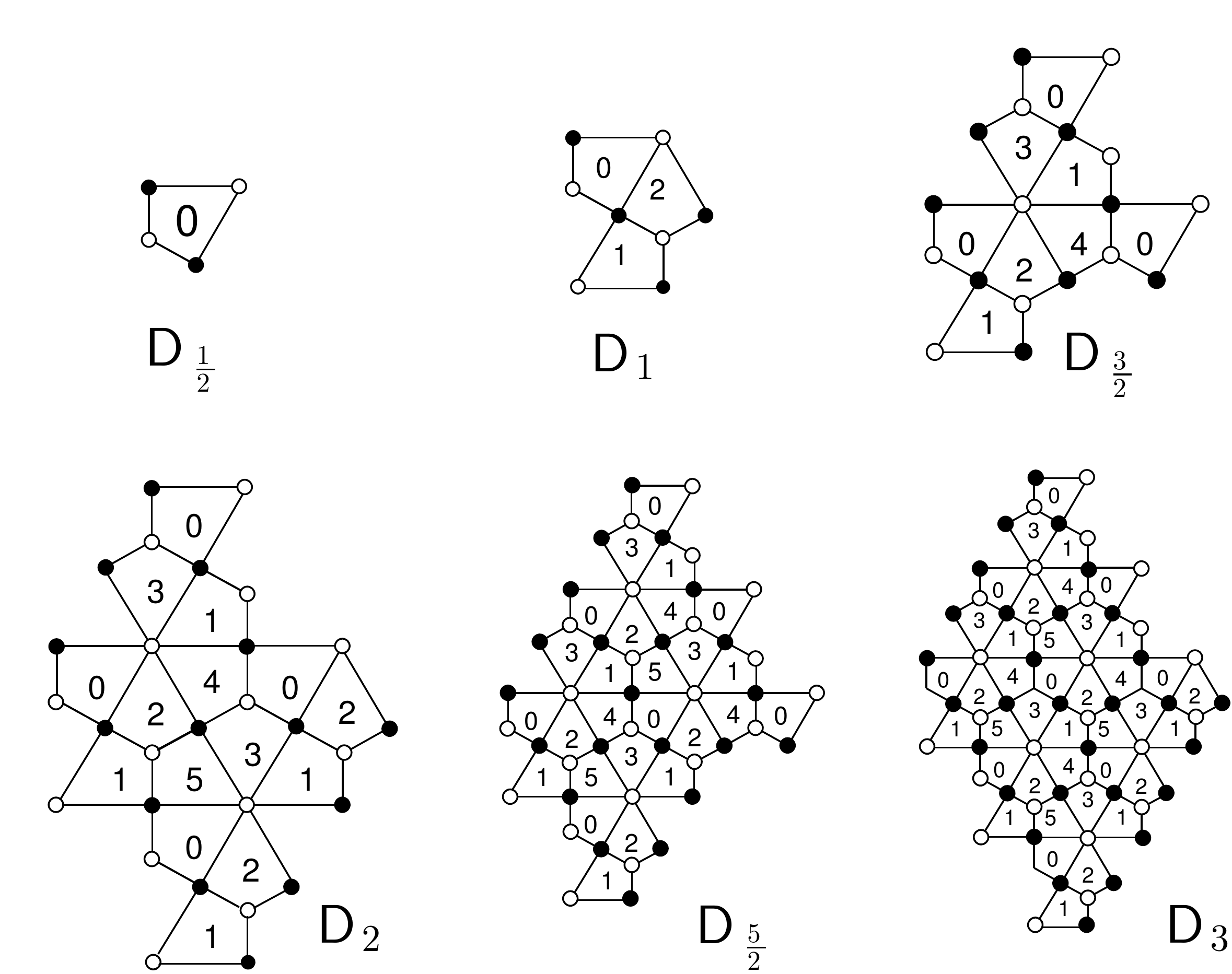}}
    \caption{\small Examples of the Aztec dragons $D_{n/2}$.}
	\label{fig: zhang diamonds}
\end{figure}

\end{subsection}

\begin{subsection}{Main Results}

We are now equipped to state our main theorem, which relates cluster variables produced by $\tau$-mutation sequences to a new family of subgraphs in the brane tiling. This family, which we call \textbf{Aztec castles}, will be defined later in Section \ref{sec:graphsintcone} and presents a generalization of the Aztec dragons introduced in \cite{CY} and studied in \cite{zhang}. 

\begin{theorem}[Main Theorem]
\label{thm:main}
For any $\tau-$mutation sequence on the dP3 quiver we can associate a subgraph $G$ of the brane tiling to each cluster variable  $y_n$ (respectively $y_n'$) produced such that 
\begin{equation*}
y_n = c(G)
\end{equation*} 
and respectively, $y_n' = c(\sigma G)$. 
\end{theorem}

In addition, in Section \ref{sec:factor}, we obtain:
\begin{itemize}
\item A factorization formula for cluster variables produced in $\tau$-mutation sequences on the dP3 quiver. This combined with the main result implies that all Aztec castles have $2^k$ perfect matchings. 
\item A recursion which enables efficient calculation of cluster-variables arising in any $\tau$-mutation sequence. 
\item Explicit formulae for the cluster variables produced by any $\tau$-mutation sequence. 
\end{itemize}

In the remainder of the introduction, we describe a construction for visualizing $\tau$-mutation sequences (Section \ref{subsec:cox}), and break the set of them into twelve types, up to symmetries (Section \ref{sec:coxsym}).  Sections \ref{sec:integercone} and \ref{sec:new-half-int} then describe the Aztec castles appearing from two of these types $\tau$-mutation sequences, corresponding to alcoves in the northeast cone and southwest cone, respectively, as defined below. The proof of our main theorem follows from the results in these sections (Theorems \ref{thm:intcone} and \ref{thm:halfintcone}) and a brief symmetry argument presented in Section \ref{sec:coxsym}.  Finally, in Section \ref{sec: conclusion}, we end with some open problems and directions for future research.

\end{subsection}

\begin{subsection}{Coxeter Lattice and Relations}
\label{subsec:cox}

We will now describe how we use the triangular lattice associated to the Coxeter group $\tilde{A_2}$, the \textbf{Coxeter lattice} \cite{coxlattice}, to allow for a clear visualization of $\tau$-mutation sequences.  This is possible because the relations observed amongst the variant $\tau_j'$'s (see Definition \ref{def:tau} and Equation \ref{eq: tau relations}) coincide with the relations amongst the generators of the affine symmetric group $\tilde{A_2}$. This is explained in Remarks \ref{rmk:taup} and \ref{rmk:tau} and also follows as an easy corollary of our main theorem. Then we can view any mutation sequence (word in the $\tau_j$'s), as a gallery of alcoves (faces in the triangular lattice) associated to $\tilde{A_2}$. Hence there is a bijection between clusters produced by $\tau$-mutation sequences and alcoves in the lattice.  

Each triangle in the lattice has edges labeled with a permutation of $\{1,2,3\}$. We fix the central shaded (purple) triangle as the origin and label its edges starting from the top going clockwise as $1,2,3$. Then we label the rest of the edges in the lattice according to the following rule: given a labeled triangle $T$, we label any triangle adjacent to $T$ by reflecting the labels over the shared edge. This gives a well-defined labeling on the entire lattice after choosing an initial labeling of the origin. Beginning at the origin (initial cluster), mutation by $\tau_i$ corresponds to moving to the adjacent alcove that shares the edge labeled $i$ with the origin. Thus, a $\tau$-mutation sequence corresponds to a walk in the lattice.

We divide this lattice into twelve regions, noting that some adjacent regions contain an overlap.  Two of these regions, which we call the \textbf{northeast} and \textbf{southwest cones}, by abuse of notation, are described more explicitly below, while the remaining four regions are rotations of the NE and SW. 

We place coordinates on these regions in terms of what we call canonical paths. These constructions will be pivotal in later sections.  Figure \ref{fig:oddeven} exhibits the lattice, with certain significant marked mutation sequences drawn as paths that split the lattice into the twelve regions.

The region labeled I (light blue)\footnote{If viewing this paper in black and white, light blue should appear white, orange should appear grey, and purple should appear dark grey.}, which we dub the \textbf{NE cone}, consists of:

\begin{enumerate}
\item All alcoves intersected by the path $123123\ldots$ after an even number of steps, including the original alcove containing the initial cluster.
\item All alcoves intersected by the path $12131213\ldots$ strictly after its second step.
\item All alcoves between those listed in the previous two bullets. 
\end{enumerate}

\noindent The region labeled VII (orange), which we dub the \textbf{SW cone}, is defined similarly to include:

\begin{enumerate}
\item All alcoves intersected by the path $321321\ldots$ after an odd number of steps.
\item All alcoves intersected by the path $32123212\ldots$ strictly after its second step.
\item All alcoves between those listed in the previous two bullets.
\end{enumerate}

The remaining regions are defined as rotations and reflections of these two cones, as according to Figure \ref{fig:oddeven} which illustrates the decomposition.   
For instance, region I overlaps with region II along certain alcoves in the path $12131213\dots$. Similarly, region $2k+1$ overlaps region $2k+2$ for $k \in \{0,1,2,3,4,5\}$ along the path $\theta(12131213\dots)$ for a permutation $\theta: \{1,2,3\}\to \{1,2,3\}$. By using symmetries of the brane tiling and Coxeter lattice, it suffices to prove our theorem in regions I (NE cone) and VII (SW cone). This is justified below in Section \ref{sec:coxsym}.

\begin{figure}[h]
    \centering
    \includegraphics[keepaspectratio=true, width=170mm]{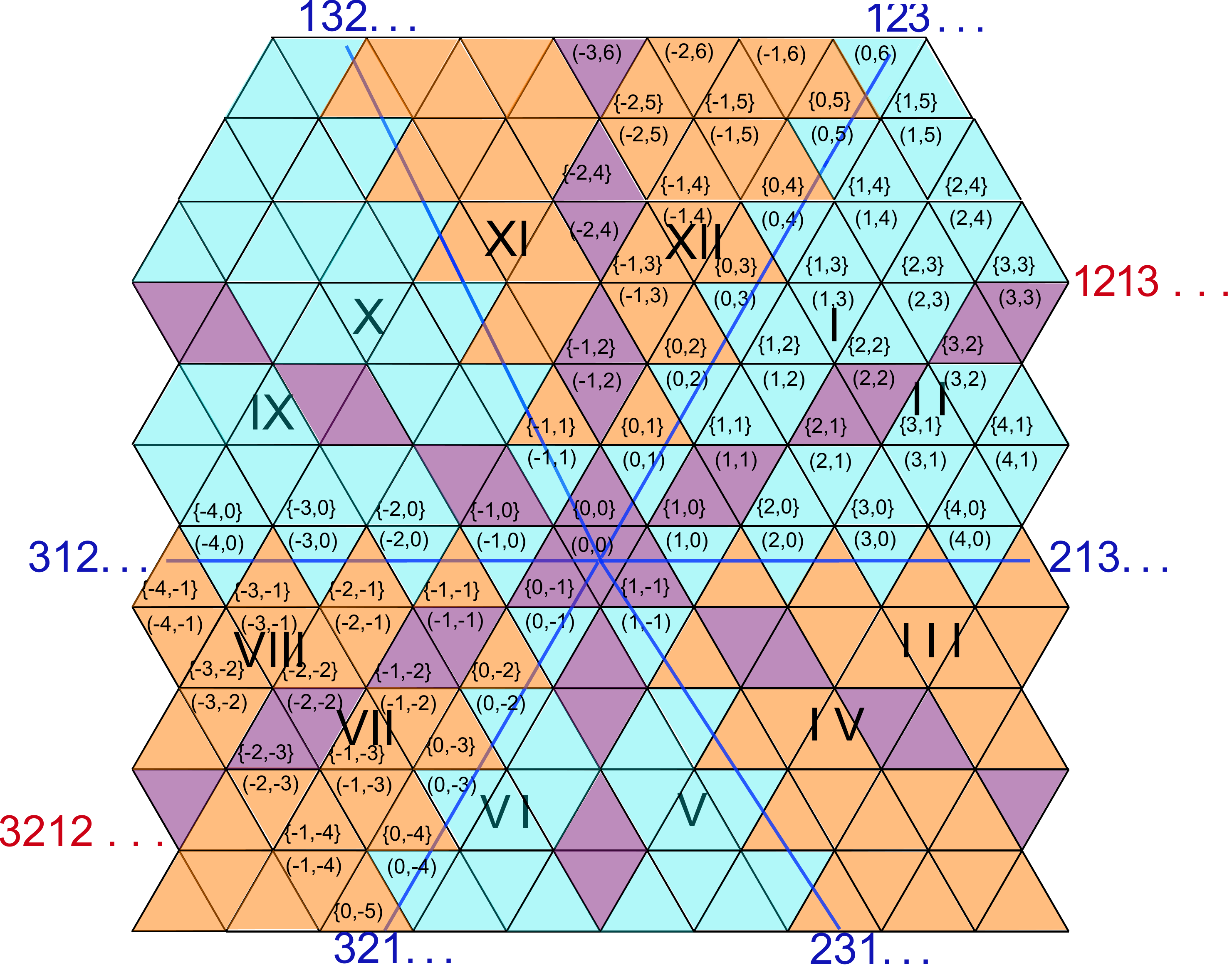}
	\caption{\small The 12 regions with alcoves labeled in the NE cone (region I), SW cone (region VII), and neighboring alcoves (e.g. in regions II, VIII, and XII).  Rotations of the NE cone (resp. SW cone) are in light blue (resp. orange).  The overlaps between I/II, III/IV, V/VI, VII/VIII, IX/X, and XI /XII, are illustrated in purple. The edges of the central hexagon $(0,0)$ are labeled starting at the top and going clockwise as 1,2,3.}
	\label{fig:oddeven}
\end{figure}

We now put a coordinate system on the alcoves of the above lattice. A \textbf{canonical path} $\mathcal{P}$ is the concatenation of two components: (1) first, a walk $\mathcal{P}_1$ along the path $(...123123...)$ in either direction (up or down) followed by (2) a horizontal component $\mathcal{P}_2$ an even number of steps in either direction (left or right). Consider an alcove $\mathcal{A}$ reached by the canonical path $\mathcal{P} = \mathcal{P}_1 \mathcal{P}_2$. Up to sign, $\mathcal{A}$ has coordinates $(i,j)$ (resp. $\{i,j\}$) if $\mathcal{P}_1$ crosses $|2j|$ (resp. $|2j-1|$) edges and $\mathcal{P}_2$ crosses $|2i|$ edges total. Signs are determined as follows: $\mathcal{P}_1$ goes up if and only if $2j \geq 0$ (resp. $2j-1 \geq 0$), $\mathcal{P}_2$ goes right if and only if $i \geq 0$. We call the alcoves reached by this process with an even (resp. odd) number of steps up or down \textbf{even alcoves} (resp. \textbf{odd alcoves}).  Note that the odd alcoves are labeled the same as the even alcoves directly south of them, except for changing $(,)$ to $\{,\}$. Also, the even alcoves are oriented the same way as the origin while odd alcoves are upside-down.  

At some height $j$ even alcove (or, equivalently, some height $j-1$ odd alcove), a canonical path may turn left or right. This turning point is indicated by a vertical bar $|$.  Let $\beta = 123$, and define $\beta^{-1} = 321$.  In this notation, the canonical paths in the NE cone ($i,j\geq 0$) have the following form based on the value of $j \mod 3$.

\begin{enumerate}
\item $\beta^{\frac{2}{3}j}|(213213\ldots)$ if $j \equiv 0 \mod 3$
\item $\beta^{\frac{2}{3}(j-1)}12|(132132\ldots)$ if $j \equiv 1 \mod 3$
\item $\beta^{\frac{1}{3}(2j-1)} 1|(321321\ldots)$ if $j \equiv 2 \mod 3$.
\end{enumerate}

\noindent On the other hand, in the SW cone, we get these same paths reflected through the origin. In this case, $i,j \leq 0$, and we get the following canonical paths.

\begin{enumerate}
\item $\beta^{\frac{2}{3}j}|(312312\ldots)$ if $j \equiv 0 \mod 3$
\item $\beta^{\frac{2}{3}(j+1)}32|(123123\ldots)$ if $j \equiv -1 \mod 3$
\item $\beta^{\frac{1}{3}(2j + 1)}3|(231231\ldots)$ if $j \equiv -2 \mod 3$ 
\end{enumerate}  

\begin{figure}[h]
    \centering
    \includegraphics[keepaspectratio=true, width=80mm]{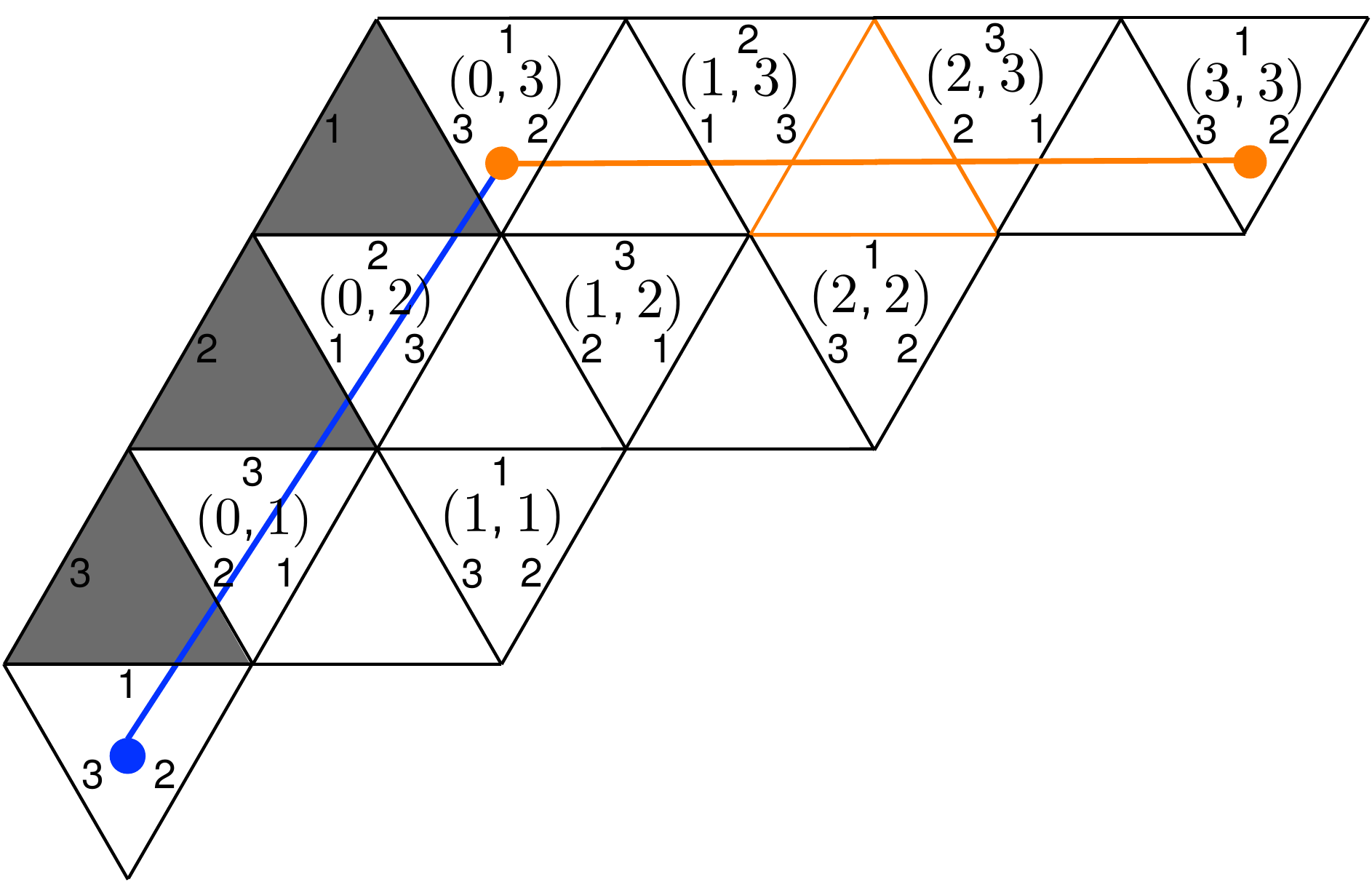} \hspace{1em}
    \includegraphics[keepaspectratio=true, width=40mm]{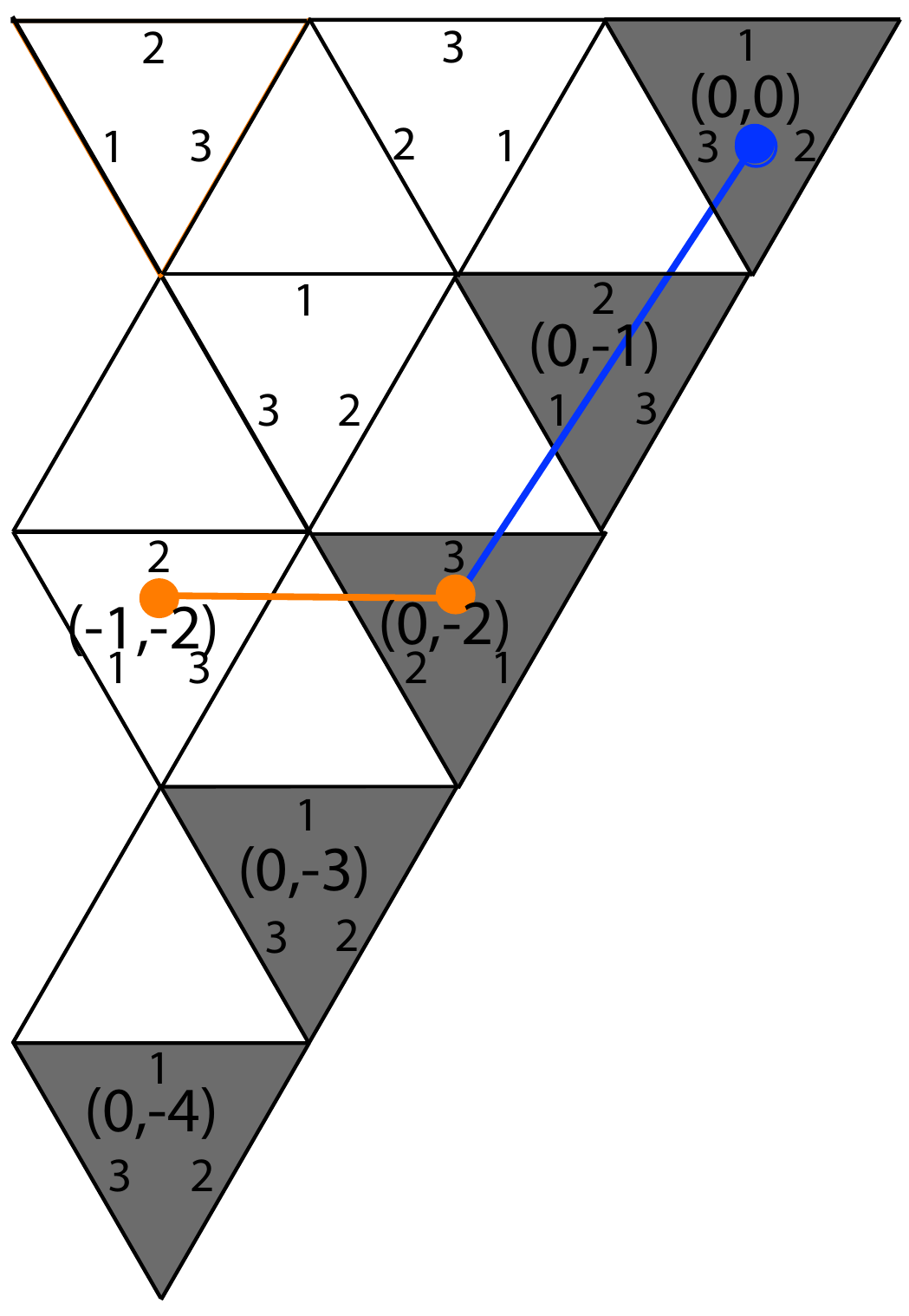}
	\caption{\small \textbf{(Left)}: The canonical path to the even alcove $(3,3)$. The blue path corresponds to the mutation sequence $\{\tau_1,\tau_2,\tau_3, \ldots\}$ and the orange horizontal path corresponds to $\{\tau_2,\tau_1,\tau_3 \dots\}$ and leads to the desired alcove.  Only the even alcoves $(i,j)$ of the NE cone (region I) are labeled in this picture. \textbf{(Right)}: The canonical path to the even alcove labeled $(-1,-2)$ in the SW cone (region VII).}
	\label{fig:canonicalpaths}
\end{figure}

\end{subsection}

\begin{subsection}{Symmetries of the Labeled Coxeter Lattice} \label{sec:coxsym}

Note that any graph isomorphism $\alpha$ of $Q$ which puts $Q$ in the form of Figure \ref{fig:quiv_brane} (up to switching directions of all the arrows) relabels vertices and hence acts naturally on mutations: $\alpha(\mu_j) = \mu_{\alpha(j)}$. Moreover, $\alpha$ acts on a cluster variable $Y$ written as a Laurent polynomial $F$ in terms of the initial seed $\{x_0, x_1, \ldots, x_5\}$ by $\alpha(Y) = F( x_{\alpha(0)}, x_{\alpha(1)}, \ldots, x_{\alpha(5)})$. Therefore, we let $\alpha$ act on an ordered cluster $\{X_0, X_1, \ldots, X_4, X_5\}$ as follows.

\begin{equation*}
\alpha \{X_0, X_1, \ldots, X_5\} = \{\alpha(X_{\alpha^{-1}(0)}), \alpha(X_{\alpha^{-1}(1)}), \ldots, \alpha(X_{\alpha^{-1}(5)})\}
\end{equation*}

\noindent In particular, $\alpha$ both changes the cluster variables and
permutes their order.

Given an alcove $\mathcal{A}$ of the lattice, we can construct a path $P$ to $\mathcal{A}$ of the form $P = \theta(P')$ where $P'$ is a canonical path in the NE or SW cone (depending on which region $\mathcal{A}$ lies in) and $\theta$ is a permutation on $\{\tau_1, \tau_2, \tau_3 \}$ that acts here on each step of $P'$. Inspection shows that for any $\theta$, there exists $\alpha$ such that $\alpha(\tau_1 \tau_2 \tau_3) = \theta(123\ldots)$. This correspondence is described in Tables \ref{tab:perms} and \ref{tab:perms2} below. 

It follows that applying $\alpha$ to the penultimate cluster variable produced at the end of the path $\theta^{-1}(P)$ yields the penultimate cluster variable produced at the end of P, possibly up to an action of $\sigma$. These penultimate cluster variables produced are not necessarily identical because the permutations $\alpha$ switch the order of some of the $\mu_i$'s. However, applying $\alpha$ with the action defined above to the ordered seed produced by $\theta^{-1}(P)$, for $\alpha$'s as listed in Table 2 (as the reader can verify for the five cases), yields the same ordered seed as that of $P$. In addition, the permutation $\theta$ sends the ordered pair $(I, VII)$ to $(R_1, R_2)$ for some antipodal regions $R_1, R_2$ in the Coxeter lattice. This implies the following Lemma that is integral to our main result.

\begin{table}[ht]

\caption{Permutations $\theta$ of $123\ldots$ and Corresponding Action on Figure \ref{fig:oddeven}} 
\centering
\begin{tabular}{ c c c c c }
\hline\hline
Region & $\theta$ & Action on Figure \ref{fig:oddeven} & $(i,j)$ sent to & $\{i,j\}$ sent to  \\ [0.5ex]
\hline
$(II, VIII)$ & (12) & {\footnotesize reflection over a line of slope $\frac{\sqrt{3}}{3}$} & $(j,i)$ & $\{j+1,i-1\}$ \\
$(IX,III)$ & (132)  & {\footnotesize $120^\circ$ rotation counter-clockwise} & $(-i-j,i)$  & $\{-i-j, i-1\}$ \\
$(X, IV)$ & (23) & {\footnotesize reflection over a vertical line} & $(-i-j, j)$ & $\{-i-j, j\}$  \\
$(V, XI)$ & (123) & {\footnotesize $120^\circ$ rotation clockwise} & $(j, -i-j)$ & $\{j+1,-i-j-1\}$ \\ 
$(VI, XII)$ & (13) & {\footnotesize reflection over a line of slope $-\sqrt{3}$} & $(i,-i-j)$ & $\{i,-i-j-1\}$  \\ [1ex] 
\hline 
\end{tabular}
\label{tab:perms} 

\end{table} 

\begin{table}[ht]

\caption{Corresponding $\alpha$ and Action on Brane Tiling} 
\centering
\begin{tabular}{ c c c c }
\hline\hline
Region & $\theta$ & $\alpha$ & Action on the brane tiling \\ [0.5ex]
\hline
$(II, VIII)$ & (12) & $(02)(13)(45)$ & {\footnotesize reflection over a line of slope $\sqrt{3}$} \\
$(IX,III)$ & (132)  & $(053142)$  & {\footnotesize $60 ^{\circ}$ rotation clockwise} \\
$(X, IV)$ & (23) & $(01)(24)(35)$ & {\footnotesize reflection over a line of slope $-\sqrt{3}$}  \\
$(V, XI)$ & (123) & $(034)(125)$ & {\footnotesize $120 ^{\circ}$ rotation clockwise} \\ 
$(VI, XII)$ & (13) & $(04)(15)$  & {\footnotesize reflection over a vertical line} \\ [1ex] 
\hline 
\end{tabular}
\label{tab:perms2} 

\end{table}

\begin{lemma}
\label{cor:zhangperm}

If we can write the cluster variables produced as $y_n = c(G)$ (respectively, $y_n' = c(\sigma G)$) for some subgraph $G$ of the brane tiling for all canonical paths in the Northeast and Southwest cones, then we may do the same for any path in the remaining regions by applying the appropriate permutation $\alpha$ or $\sigma \alpha$ to $G$. 
\end{lemma}

\end{subsection}
\end{section}

\begin{section}{Explicit Formulae for Cluster Variables}
\label{sec:factor}

The main result of this section is an explicit formula for the cluster corresponding to any alcove of the Coxeter lattice described in Section \ref{sec:introduction} (and thus for any $\tau$-mutation sequence).   

\begin{theorem}
\label{thm:alcove}

Let $A = \frac{x_2x_4+x_3x_5}{x_0x_1}, B = \frac{x_1x_4+x_0x_5}{x_2x_3}, C = \frac{x_0x_2+x_1x_3}{x_4x_5}.$ Let $\mathcal{C}_{(i,j)}$ (resp. $\mathcal{C}_{\{i,j)\}}$) be the ordered cluster  corresponding to the alcove $(i,j)$ (resp. $\{i,j\}$).  
By ordered cluster, we mean that the $k$th entry of the cluster is the variable at vertex $k$. For brevity, we represent $\mathcal{C}_{(i,j)}$ and $\mathcal{C}_{\{i,j\}}$ in the form $\{v_0, v_2, v_4\}$ where $v_{i}$ is the variable at vertex $i$ (recall that $v_{i+1} = \sigma(v_{i})$ for $i = 0,2,4)$. Also, to clean up the presentation, let $E_{\{A,B,C\}}[f,g,h] = A^f B^g C^h$. Then $\mathcal{C}_{(\frac{i}{2},j)} =$

\begin{itemize}

\item $\{x_0 E_{\{A,B,C\}}[\frac{i^2}{12}+\frac{i j}{6}+\frac{j^2}{3}, ~\hspace{5em}~ \frac{i}{6}+\frac{i^2}{12}-\frac{j}{3}+\frac{i j}{6}+\frac{j^2}{3}, ~\hspace{5em}~ -\frac{i}{6}+\frac{i^2}{12}-\frac{2 j}{3}+\frac{i j}{6}+\frac{j^2}{3}], \\ x_2 E_{\{A,B,C\}} [-\frac{i}{3}+\frac{i^2}{12}-\frac{j}{3}+\frac{i j}{6}+\frac{j^2}{3},~\hspace{5em}~ \frac{i^2}{12}+\frac{i j}{6}+\frac{j^2}{3}, ~\hspace{5em}~ -\frac{i}{3}+\frac{i^2}{12}-\frac{j}{3}+\frac{i j}{6}+\frac{j^2}{3} ], \\ x_4 E_{\{A,B,C\}}[\frac{i^2}{12}+\frac{i j}{6}+\frac{j^2}{3}, ~\hspace{5em}~ \frac{i}{3}+\frac{i^2}{12}+\frac{j}{3}+\frac{i j}{6}+\frac{j^2}{3}, ~\hspace{5em}~ \frac{i^2}{12}+\frac{i j}{6}+\frac{j^2}{3} ]\}$ \\ if $|i| = 0 \mod 6$, $|j| = 0 \mod 3$;
   
\item $\{x_0 E_{\{A,B,C\}}[\frac{1}{12}(i^2+2 i (-1+j)+4 (1+j+j^2)), ~\hspace{5em}~  \frac{1}{12} (-4+i^2+2 i j+4 j^2), \\ \frac{1}{12} (i^2+2 i (-2+j)+4 (-1+j) j) ], \\
   x_{2} E_{\{A,B,C\}}[\frac{1}{12} (i^2+2 i (1+j)+4 j (2+j)), ~ ~  \frac{1}{12} (i^2+2 i (2+j)+4 (1+j+j^2)), \\ \frac{1}{12} (-4+i^2+2 i j+4 j^2)], \\
   x_{4} E_{\{A,B,C\}}[\frac{1}{12} (8+i^2+2 i j+4 j^2)-1,  ~\hspace{5em}~ \frac{1}{12} (i^2+4 (-1+j) j+2 i (1+j)),  \\ \frac{1}{12} (i^2+2 i (-1+j)+4 (-1+j)^2) ] \}$ \\ if $|i| = 0 \mod 6, |j| = 1 \mod 3$;
      
\item $\{ x_0 E_{\{A,B,C\}}[\frac{1}{12} \left(i^2+2 i (1+j)+4 (1+j)^2\right),  ~\hspace{5em}~ \frac{1}{12} \left(i^2+4 j (1+j)+2 i (2+j)\right), \\ \frac{1}{12} \left(-4+i^2+2 i j+4 j^2\right) ], \\
    x_2 E_{\{A,B,C\}}[\frac{1}{12} \left(-4+i^2+2 i j+4 j^2\right), ~\hspace{5em}~ \frac{1}{12} \left(i^2+2 i (1+j)+4 \left(1-j+j^2\right)\right), \\ \frac{1}{12} \left(i^2+2 i (-1+j)+4 (-2+j) j\right) ], \\
    x_4 E_{\{A,B,C\}}[\frac{1}{12} \left(i^2+2 i (-1+j)+4 j (1+j)\right),  ~\hspace{5em}~ 1/12 (-4 + i^2 + 2 i j + 4 j^2), \\ \frac{1}{12} \left(i^2+2 i (-2+j)+4 \left(1-j+j^2\right)\right) ] \} $ \\ if $|i| = 0 \mod 6, |j| = 2 \mod 3$;
   
\end{itemize}

and $\mathcal{C}_{\{\frac{i}{2} + \frac{1}{2},j-1\}} =$

\begin{itemize}

\item $\{x_1 E_{\{A,B,C\}}[\frac{1}{4}+\frac{i^2}{12}+\frac{j}{2}+\frac{i j}{6}+\frac{j^2}{3},  ~\hspace{5em}~ -\frac{1}{4}+\frac{i}{6}+\frac{i^2}{12}+\frac{j}{6}+\frac{i j}{6}+\frac{j^2}{3},  ~\hspace{5em}~ -\frac{1}{4}-\frac{i}{6}+\frac{i^2}{12}-\frac{j}{6}+\frac{i j}{6}+\frac{j^2}{3}], \\
   x_3 E_{\{A,B,C\}}[-\frac{1}{4}-\frac{i}{6}+\frac{i^2}{12}-\frac{j}{6}+\frac{i j}{6}+\frac{j^2}{3},  ~\hspace{5em}~ \frac{1}{4}+\frac{i^2}{12}-\frac{j}{2}+\frac{i j}{6}+\frac{j^2}{3},  ~\hspace{5em}~ \frac{1}{4}-\frac{i}{3}+\frac{i^2}{12}-\frac{5 j}{6}+\frac{i j}{6}+\frac{j^2}{3} ], \\
   x_5 E_{\{A,B,C\}}[-\frac{1}{4}+\frac{i}{6}+\frac{i^2}{12}+\frac{j}{6}+\frac{i j}{6}+\frac{j^2}{3},  ~\hspace{5em}~ \frac{1}{4}+\frac{i}{3}+\frac{i^2}{12}-\frac{j}{6}+\frac{i j}{6}+\frac{j^2}{3},  ~\hspace{5em}~ \frac{1}{4}+\frac{i^2}{12}-\frac{j}{2}+\frac{i j}{6}+\frac{j^2}{3}] \}$ \\ if $|i| = 3 \mod 6, |j| = 0 \mod 3$;
          
\item $\{ x_1 E_{\{A,B,C\}}[\frac{1}{12} \left(1+i^2+2 i (-1+j)-2 j+4 j^2\right),  ~\hspace{5em}~ \frac{1}{12} \left(-1+i^2-6 j+2 i j+4 j^2\right), \\ \frac{1}{12} \left(3+i^2+2 i (-2+j)-10 j+4 j^2\right)],  \\
    x_3 E_{\{A,B,C\}}[\frac{1}{12} \left(-3+i^2+2 j+4 j^2+2 i (1+j)\right),  ~\hspace{5em}~ \frac{1}{12} \left(7+i^2-2 j+4 j^2+2 i (2+j)\right), \\ \frac{1}{12} \left(-1+i^2-6 j+2 i j+4 j^2\right)], \\
    x_5 E_{\{A,B,C\}}[\frac{1}{12} \left(11+i^2+6 j+2 i j+4 j^2\right)-1,  ~\hspace{5em}~ \frac{1}{12} \left(-3+i^2+2 j+4 j^2+2 i (1+j)\right), \\ \frac{1}{12} \left(1+i^2+2 i (-1+j)-2 j+4 j^2\right)] \}$ \\ if $|i| = 3 \mod 6, |j| = 1 \mod 3$;
    
\item $\{x_1 E_{\{A,B,C\}}[\frac{1}{12} \left(1+i^2+2 j+4 j^2+2 i (1+j)\right),  ~\hspace{5em}~ \frac{1}{12} \left(3+i^2-2 j+4 j^2+2 i (2+j)\right), \\ \frac{1}{12} \left(-1+i^2-6 j+2 i j+4 j^2\right)], 
\\ x_3 E_{\{A,B,C\}}[\frac{1}{12} \left(-1+i^2+6 j+2 i j+4 j^2\right),  ~\hspace{5em}~ \frac{1}{12} \left(1+i^2+2 j+4 j^2+2 i (1+j)\right), \\ \frac{1}{12} \left(-3+i^2+2 i (-1+j)-2 j+4 j^2\right)], 
\\ x_5 E_{\{A,B,C\}}[\frac{1}{12} \left(-3+i^2+2 i (-1+j)-2 j+4 j^2\right),  ~\hspace{5em}~ \frac{1}{12} \left(-1+i^2-6 j+2 i j+4 j^2\right), \\ \frac{1}{12} \left(7+i^2+2 i (-2+j)-10 j+4 j^2\right)]\}$ \\ if $|i| = 3 \mod 6, |j| = 2 \mod 3$.

\end{itemize}

\end{theorem}

\begin{remark} These expressions allow us to express all cluster variables which appear in some alcove.  In particular, the cluster variables explicitly described by the above theorem are for any alcove reached after (1) $2\ell$ steps along $(\ldots 123 \ldots)$ (up or down) and (2) $3k$ horizontal steps right or left. Hence, every third alcove to the right or left of $(\ldots 123 \ldots)$ (and hence on any row of the Coxeter lattice) is understood. Since a horizontal path is given by $\theta(123\ldots)$ for some permutation $\theta$, the cluster variable at vertex $i$ is updated only after every third step. Hence, the above formulas are sufficient for extracting the cluster variable for any alcove.
\end{remark}

\begin{remark}
\label{rmk:tau}
It is straightforward but tedious to extract from the explicit formula that the relations from Equation \ref{eq: tau relations} are the only relations satisfied by the $\tau'$'s. To see this, for each alcove $(i,j)$ (resp. $\{i,j\}$), set the expressions for $\mathcal{C}_{(i,j)}$ (resp. $\mathcal{C}_{\{i,j\}}$) equal to the inital cluster $\{x_0, x_1, x_2, x_3, x_4, x_5 \}$ (considered up to permutation) and solve for $i$ and $j$. In all cases we will get $i = j = 0$ as desired. 
\end{remark}

Theorem \ref{thm:alcove} is driven by the fact that all cluster variables produced by $\tau$-mutation sequences factor. To prove this we introduce some convenient notation. Suppose that we have a $\tau$-mutation sequence $S = s_1 s_2 s_3 s_4 \ldots$. Denote by $\eta(n,i)$ the number of times we have mutated at vertex $i$ in the first $n$ terms ($2n$ total mutations) of $S$. Furthermore, let $\mathcal{C}_n(Q)$ denote the ordered cluster obtained after applying $n$ terms of $S$. Define the following six functions. For $i = 0,2,4$, we set
\begin{equation*}
f_i(n) =
  \begin{cases}
   i & \text{if } \eta(n,i) = 0\, \mathrm{mod}\, 2 \\
   i+1       & \text{else}
  \end{cases}
\end{equation*}
and for $i = 1,3,5$, we set
\begin{equation*}
f_i(n) =
  \begin{cases}
   i & \text{if } \eta(n,i) = 0\, \mathrm{mod}\, 2 \\
   i-1       & \text{else}.
  \end{cases}
\end{equation*}
\begin{proposition}[Factorization Phenomenon]
\label{thm:factor}
Fix $S = s_1 s_2 s_3 s_4 \ldots$ to be any $\tau- $mutation sequence on $Q$. The cluster $\mathcal{C}_n(Q)$ has the following form:
\begin{align*}
\mathcal{C}_n(Q) = \{ x_{f_0(n)}A^{a_1}B^{b_1}C^{c_1}, x_{f_1(n)}A^{a_1}B^{b_1}C^{c_1}, x_{f_2(n)}A^{a_2}B^{b_2}C^{c_2},\\ x_{f_3(n)}A^{a_2}B^{b_2}C^{c_2}, x_{f_4(n)}A^{a_3}B^{b_3}C^{c_3},  x_{f_5(n)}A^{a_3}B^{b_3}C^{c_3} \}.
\end{align*}
\end{proposition}
\begin{proof}
\begin{par}Call the above form for the cluster variables a $\tau-$presentation, and the $x_{f_i(n)}$ term in each factorization the \textbf{leading term}. We induct on the cluster $\mathcal{C}_n(Q)$. After applying any initial $\tau_i$ the cluster has the above form, thus proving our base case. Assume that $\mathcal{C}_n(Q)$ has a $\tau$-presentation and mutate at vertex $0$. Denote the cluster variable at vertex $i$ in $\mathcal{C}_n(Q)$ by $X_i$. It suffices to show that $X_0',$ the variable produced at vertex $0$ after applying $\tau_1$, has $\tau-$presentation $x_{f_0(n+1)}A^{a_1'}B^{b_1'}C^{c_1'}$; the cases of mutating at other vertices are analogous. \end{par}
\begin{par}Note that by the symmetry of the quiver,  $\theta(X_i) = X_{\theta(i)}$ for any permutation $\theta: \{0,1,2,3,4,5\} \to \{0,1,2,3,4,5\}$. Each $\tau_i$ has an action on the quiver that simply relabels the vertices by switching vertex $2i-1$ with vertex $2i-2$. Moreover, when we mutate by  $\tau_i$, $\eta (n,2i-1)$ and $\eta (n,2i-2)$ increase by $1$, which by the inductive assumption flips the leading terms of the cluster variables at vertices $2i-1$ and $2i-2$. Since each vertex $i$ has leading term $x_i$ in the initial cluster, for the purposes of the exchange at vertex $0$, we can assume without loss of generality each cluster variable in $\mathcal{C}_n(Q)$ has $f_i(n) = i$, and that the quiver has the original labeling. This is because the exchange polynomial at any vertex is invariant under $\tau$-mutation by the described action of any $\tau_i$ on $Q$. \end{par}

Now suppose $\mathcal{C}_n(Q)$ has the form as in the statement of this proposition. Mutating at vertex $0$, we compute the exchange relation: 
\[
\begin{split}
&X_0X_0' = X_2X_4+X_3X_5 \implies \\ &X_0' X_0= \left(\frac{x_2x_4+x_3x_5}{x_0x_1}\right)^{a_2+a_3}\left( \frac{x_1x_4+x_0x_5}{x_2x_3}\right)^{b_2+b_3}\left( \frac{x_0x_2+x_1x_3}{x_4x_5}\right)^{c_2+c_3}\left(x_2x_4+x_3x_5\right) \implies \\ 
&X_0' = x_1\left(\frac{x_2x_4+x_3x_5}{x_0x_1}\right)^{a_2+a_3-a_1+1}\left( \frac{x_1x_4+x_0x_5}{x_2x_3}\right)^{b_2+b_3-b_1}\left( \frac{x_0x_2+x_1x_3}{x_4x_5}\right)^{c_2+c_3-c_1}.
\end{split}
\]
So $X_0'$ has $\tau-$presentation with leading term $x_{f_0(n+1)} = x_1$, as desired. This completes the induction. Finally, note that all exponents in the $\tau$-presentation are non-negative by the Laurent phenomenon \cite{FZ}. 
\end{proof}

Using Proposition \ref{thm:factor}, we can analyze any sequence of mutations by looking at the corresponding sequence of exponents $\{a,b,c\}$ associated to the cluster variables. Using symmetry and noting how the exponents of $X_0'$ relate to the exponents of variables in the previous cluster yields a recursion for computation of cluster variables that can be implemented in linear time (we omit the precise statement for brevity). Using this recursion and induction, we can verify the following explicit formula for cluster variables arising from the $123123 \ldots$ sequence.

\begin{proposition}
\label{cor:beta}
For positive $n$, recall that $\beta^n = (123)^n$. For negative $n$, we define $\beta^{n} = (321)^{-n}$. We have two cases:

\begin{itemize}

\item[(a)] If $n$ is even, 
\begin{equation*}
\label{eqn:exdiam1}
\begin{split}
\beta^n(Q) = \mathcal{C}_{3n}(Q) = \{&x_0 A^{\frac{3n^2}{4}} B^{\frac{3n^2 - 2n}{4}} C^{\frac{3n^2 - 4n}{4}}, x_1 A^{\frac{3n^2}{4}} B^{\frac{3n^2 - 2n}{4}} C^{\frac{3n^2 - 4n}{4}},\\
 &x_2 A^{\frac{3n^2+2n}{4}} B^{\frac{3n^2}{4}} C^{\frac{3n^2 - 2n}{4}}, 
 x_3 A^{\frac{3n^2+2n}{4}} B^{\frac{3n^2}{4}} C^{\frac{3n^2 - 2n}{4}},\\
  &x_4 A^{\frac{3n^2+4n}{4}} B^{\frac{3n^2+2n}{4}} C^{\frac{3n^2}{4}}, x_5 A^{\frac{3n^2+4n}{4}} B^{\frac{3n^2+2n}{4}} C^{\frac{3n^2}{4}}\}
\end{split}
\end{equation*}

\item[(b)] If $n$ is odd,
\begin{equation*}
\begin{split}
\beta^n(Q) = \mathcal{C}_{3n}(Q) = \{&x_1 A^{\frac{3n^2+1}{4}} B^{\frac{3n^2 - 2n - 1}{4}} C^{\frac{3n^2 - 4n + 1}{4}}, x_0 A^{\frac{3n^2 + 1}{4}} B^{\frac{3n^2 - 2n - 1}{4}} C^{\frac{3n^2 - 4n + 1}{4}}, \\
&x_3 A^{\frac{3n^2+2n-1}{4}} B^{\frac{3n^2 + 1}{4}} C^{\frac{3n^2 -2n - 1}{4}}, x_2 A^{\frac{3n^2+2n-1}{4}} B^{\frac{3n^2 + 1}{4}} C^{\frac{3n^2 -2n - 1}{4}},\\
 &x_5 A^{\frac{3n^2+4n+1}{4}} B^{\frac{3n^2+2n-1}{4}} C^{\frac{3n^2+1}{4}}, x_4 A^{\frac{3n^2+4n+1}{4}} 
B^{\frac{3n^2+2n-1}{4}} C^{\frac{3n^2+1}{4}}\}.
\end{split}
\label{eqn:exdiam2}
\end{equation*}

\end{itemize}

\end{proposition}

Now we have all the ingredients to provide the method of proof of Theorem \ref{thm:alcove}. We show this explicitly to calculate the variable at vertex $0$ in the cluster $\mathcal{C}_{ij}$ with $|i| \equiv 0 \mod 6, |j| \equiv 0 \mod 3$. The canonical path to $(\frac{i}{2},j)$ is given by $(123)^{\frac{2j}{3}} (213)^{\frac{i}{3}}$. The cluster variable at vertex $0$ after applying $(213)^{\frac{i}{3}}$, i.e. $\big(\theta(123)\big)^{\frac{i}{3}}$ for $\theta = (12)$, is 
\begin{equation*}
\label{eqn:clus1}
\alpha (x_ 2 A^{\frac{1}{4} \left(\frac{2 i}{3}+\frac{i^2}{3}\right)} B^{ \frac{i^2}{12}} C^{\frac{1}{4} \left(-\frac{2 i}{3}+\frac{i^2}{3}\right)}) = x_0A^{\frac{i^2}{12}}B^{\frac{1}{4} \left(\frac{2 i}{3}+\frac{i^2}{3}\right)}C^{\frac{1}{4} \left(-\frac{2 i}{3}+\frac{i^2}{3}\right)}
\end{equation*} 

\noindent where $\alpha$ is the permutation $(02)(13)(45)$ obtained from Table \ref{tab:perms2}. We read the cluster after $(123)^{\frac{2j}{3}}$ from Proposition \ref{cor:beta}. 

Denote the cluster variables by $\{X_0, X_1, X_2, X_3, X_4, X_5 \}$. Then we have:

\begin{align*}
X_0 = x_0A^{\frac{j^2}{3}}B^{\frac{j^2-j}{3}}C^{\frac{j^2-2j}{3}},
X_1 =  x_1A^{\frac{j^2}{3}}B^{\frac{j^2-j}{3}}C^{\frac{j^2-2j}{3}},
X_2 =  x_2A^{\frac{j^2+j}{3}}B^{\frac{j^2}{3}}C^{\frac{j^2-j}{3}}, \\
X_3 =  x_3A^{\frac{j^2+j}{3}}B^{\frac{j^2}{3}}C^{\frac{j^2-j}{3}}, 
X_4 = x_4A^{\frac{j^2+2j}{3}}B^{\frac{j^2+j}{3}}C^{\frac{j^2}{3}},
X_5 =  x_5A^{\frac{j^2+2j}{3}}B^{\frac{j^2+j}{3}}C^{\frac{j^2}{3}}.
\end{align*}

Substituting $X_i$ into $x_i$ in the formula for the cluster variable at vertex $0$, we obtain the desired computation of the variable at vertex zero of the cluster $\mathcal{C}_{ij}$:
\begin{equation*}
x_0 A^{\frac{i^2}{12}+\frac{i j}{6}+\frac{j^2}{3}} B^{\frac{i}{6}+\frac{i^2}{12}-\frac{j}{3}+\frac{i j}{6}+\frac{j^2}{3}} C^{-\frac{i}{6}+\frac{i^2}{12}-\frac{2 j}{3}+\frac{i j}{6}+\frac{j^2}{3}}.
\end{equation*}   

\end{section}

\begin{section}{The Northeast Cone}
\label{sec:integercone}

Now we return to the original problem of relating cluster variables produced by $\tau$-mutation sequences to subgraphs of the brane tiling, starting with the NE cone of the Coxeter lattice first introduced in Section \ref{subsec:cox}. We will first describe a shorthand method for drawing Aztec castles and provide several examples. We will then proceed to prove a recursion on the weights of perfect matchings of these subgraphs utilizing a technique known as graphical condensation first introduced in \cite{kuo} and further generalized in \cite{speyer}. Finally we will prove a recursion on the covering monomials of the castles, allowing us to complete the proof of the Northeast Cone Theorem (Theorem \ref{thm:intcone}).

\begin{subsection}{Hexagonal Notation}

\par{In this section we introduce convenient notation that allows us to define the graphs of interest. Given a six-tuple $(a,b,c,d,e,f)$ of non-negative integers satisfying \begin{eqnarray} \label{eq:closed_curve1} a+c &=& d+f \\ \label{eq:closed_curve2} b+c &=& e+f,\end{eqnarray}                                                                                                                                                                        we associate the following hexagon illustrated in Figure \ref{fig:hexagon}.  As we will explain, these two equations are necessary and sufficient to ensure that the hexagon will be a closed curve.

\begin{figure}[h]
    \centering
    \includegraphics[keepaspectratio=true, width=65mm]{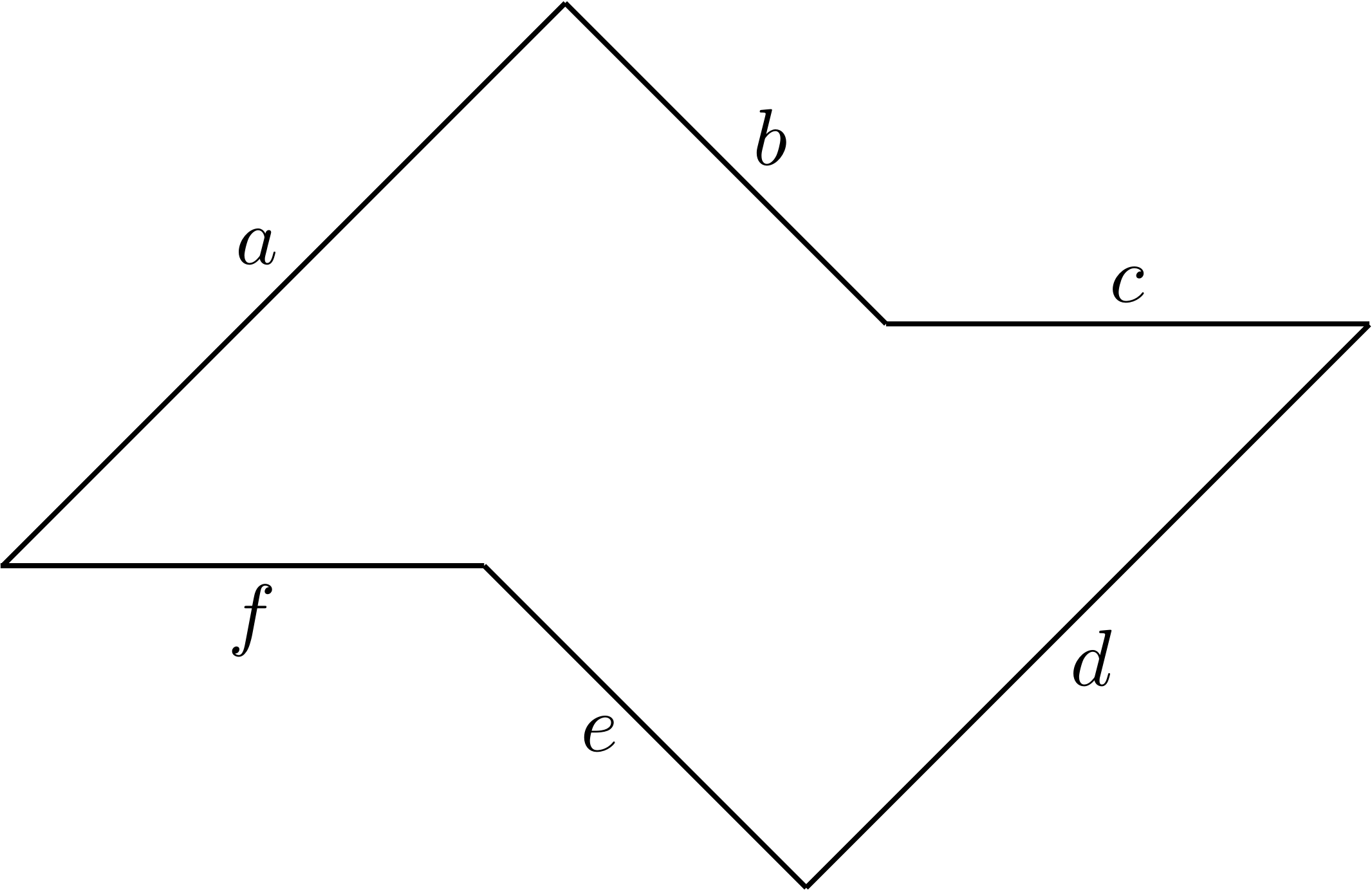}
    \caption{\small The hexagonal notation for a general castle, the sides of which are identified by the tuple $(a,b,c,d,e,f)$.  }
    \label{fig:hexagon}
\end{figure}

}

\par{
In turn, this six-tuple defines a subgraph $G$ of the brane tiling. We use the tuple to define the boundary of the planar graph $G$ in terms of a path in the brane tiling. Begin the boundary path $P$ at some point $p$, called the basepoint of $G$, in the brane tiling which is the white vertex in the center of a $6$-cycle (see Section \ref{subsec:quivbrane} for terminology). The boundary path consists of six subpaths, each of which will be defined in terms of the cardinal directions naturally associated to the edges of the brane tiling. 

}

For notational convention, if $P = (p_1, p_2, p_3, \ldots)$ and $Q = (q_1, q_2, q_3, \ldots)$ are paths in the brane tiling represented by cardinal directions $p_1, p_2,  p_3, \ldots, p_n$ and $q_1, q_2, q_3, \ldots,  q_m$,
respectively, denote the left-to-right concatenation of the two paths by $PQ = (p_1, p_2, p_3, \ldots, p_n, q_1, \ldots, q_m)$. Likewise, $P^i$ denotes the path $P$ concatenated with itself $i$ times. Furthermore, define the following four-step subpaths illustrated in Figure \ref{fig: boundarypaths}.  Note that these are zig-zag paths \cite{GK} in the brane tiling, as observed by Rick Kenyon, although we found these graphs by a different technique\footnote{A brief discussion on how these graphs (and the SW Aztec castles described in Section \ref{sec:new-half-int}) were first computed can be found in Problem \ref{prob:ef} where we utilised a method detailed in \cite{franco_eager}. Eager and Franco's methods were generalizations of those from \cite{CJ}, based on earlier work on ``pyramid partition functions'' for the conifold of Szendr\H{o}i \cite{szendroi}.}.

\begin{figure}[H]
    \centering
    \includegraphics[keepaspectratio=true, width=90mm]{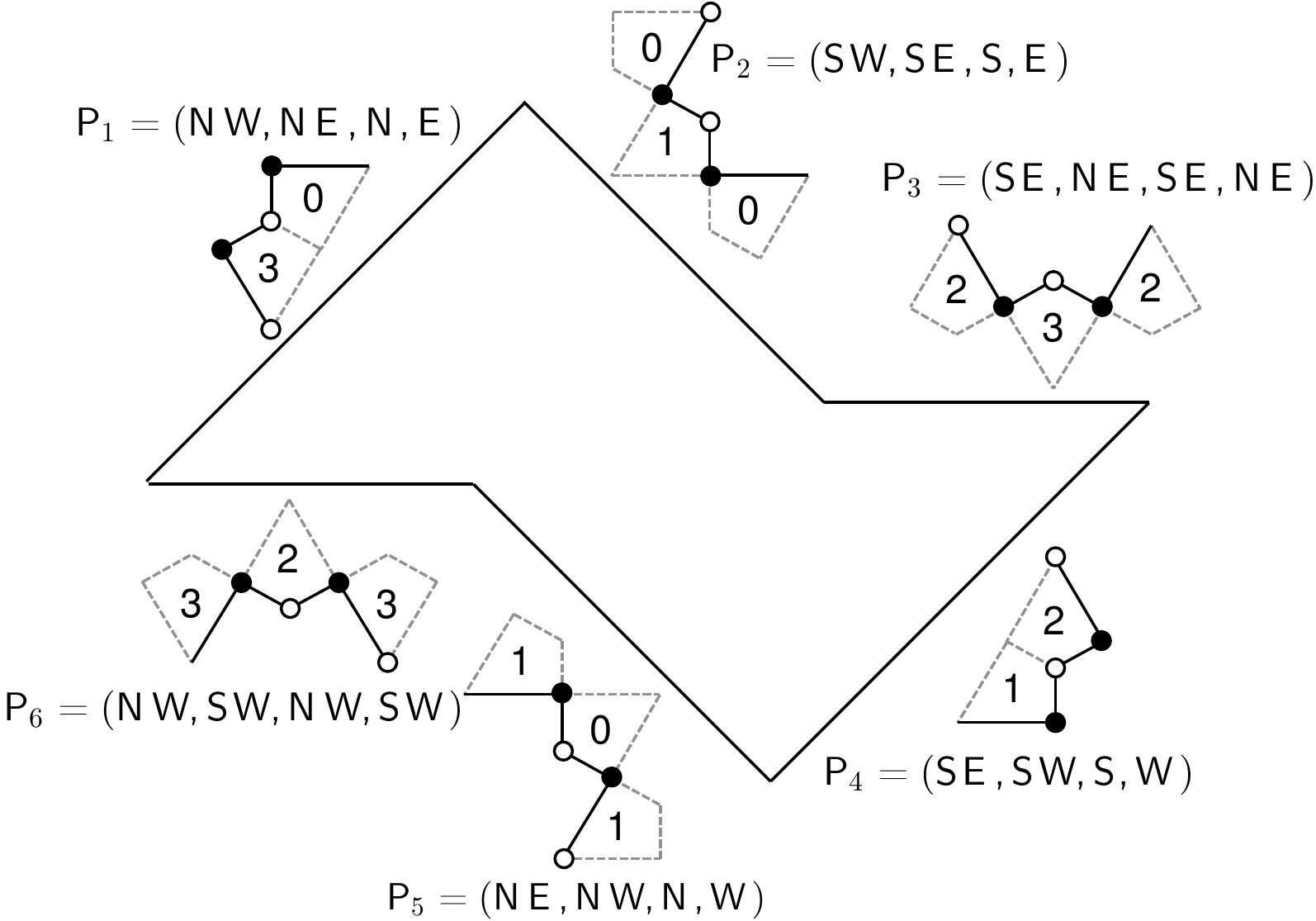}
    \caption{\small The six four-step subpaths that trace out a boundary path P, placed around a graph in hexagonal notation to signify their position in the brane tiling. The desired subgraph is induced by all of the vertices bounded by P.}
    \label{fig: boundarypaths}
\end{figure}

Now, the tuple $(a,b,c,d,e,f)$ corresponds to the following path. Choose $p$ as described above and trace the path 
\[
P = P_1^a P_2^b P_3^c P_4^d P_5^e P_6^f.
\]
Provided that $P$ is a simple closed curve, we define $G$ to be the subgraph induced by all vertices on the boundary path $P$ and in the interior region of $P$. When we want to keep track of the basepoint $p$, it is convenient to write $G = (a,b,c,d,e,f)_p$.  To see that equations (\ref{eq:closed_curve1}) and (\ref{eq:closed_curve2}) are necessary and sufficient, let the direction of $a$ represent one fundamental loop of the torus and $b$ represent the other.  Then the above directions correspond as follows. $$a \leftrightarrow (1,0), b \leftrightarrow (0,1),  c \leftrightarrow (1,1), d \leftrightarrow (-1,0), e \leftrightarrow (0,-1), f \leftrightarrow (-1,-1)$$ in the homology basis of the torus.  Hence, $(a,b,c,d,e,f)$ corresponds to a closed curve if and only if $(a+c-d-f,b+c-e-f) = (0,0)$.

\end{subsection}

\begin{subsection}{Graphs in the Northeast Cone}
\label{sec:graphsintcone}

Utilizing this hexagonal notation, we now define a new family of subgraphs of the brane tiling called \textbf{NE Aztec castles}. Note that these castles include the integer order Aztec dragons as a special case. Later we will prove that this family consists of all of the subgraphs of the brane tiling contained in the NE cone.
\footnote{The Aztec castles also appear in \cite{lai} as the dual graphs to a more general family of ``Aztec dragon regions'' written $DR^{(1)}(a,b,c)$  and $DR^{(2)}(a,b,c)$ in his notation, corresponding to SW Aztec castles and NE Aztec castles, respectively, for special choices of $a,b$, and $c$. Tri Lai computes their perfect matchings under a different weighting system using techniques from \cite{lai'}.}

\begin{definition}[Northeast Aztec Castle]
\label{def:integer castle}
Define the $(i, j)$--NE Aztec castle $\gamma^{j}_i$ for $i,j \in \mathbb{Z}_{\geq 0}$ as follows.  These correspond to the even alcoves in the 
NE cone (up to a reflection by $\sigma$ as defined in Section \ref{sec:tau}).  
\begin{equation}
\gamma^{j}_i = 
(i+j,~~j,~~i-1,~~i+j-1,~~j-1,~~i)
\end{equation}

\end{definition}

\begin{remark}
Notice that for $0 \leq i \leq j$, we can recover $i$ and $j$ from hexagonal notation $(a,b,c,d,e,f)$ as $i=c$ and $j=e$ (when $i \equiv 0 \mod 3$) so that the corresponding Aztec castle is $\sigma \gamma_i^j$ in this case, and as $i=f$ and $j=b$ for $\gamma_i^j$ otherwise.
\end{remark}

\end{subsection}

\begin{subsection}{Integer Order Aztec Dragons}

The integer order Aztec dragons introduced in \cite{zhang} fit into this framework with a slight modification. In hexagonal notation, the $n^{\mathrm{th}}$ order Aztec dragon, $D_n$ ($n \in \mathbb{N}$) corresponds to the alcove $(0,n)$ and thus would be labeled as $\sigma \gamma_0^n = (n-1,n-1,0,n,n,-1)$.  (Note that $i \equiv 0 \mod 3$ in this case.)  We trace out all of the paths $P_1$, $P_2$, $\dots$ as usual, until we come to $P_6$.  The significance of the $-1$ is:

(a) Trace the path $P_6$ backwards, i.e. (NE, SE, NE, SE) bordering faces $1$ and $5$.

(b) We then remove the faces $1$ and $5$ since they border a ``negative'' path.

\noindent With this modification, we indeed obtain the integer order Aztec dragons as desired.  

\vspace{1em}

We do the same procedure when $b$, $c$, or $e$ equal $-1$ in hexagonal notation.  By construction, it is not possible for $a$ nor $d$ to be negative for $i,j\geq 0$. 

\begin{remark}
\label{rem:neghex}

We will not use or need this in this paper, but there is a more general way to interpret negatives, such as $-k$ for $k>1$, in hexagonal notation.  Again, a negative value signifies tracing a subpath backwards but when $k>1$, we remove more exterior faces from the graph.  Note that such values do not appear for $i,j\geq 0$.  

Using this construction, we can describe the graphs corresponding to alcoves in the SW cone.  In particular, the Aztec castle $\tilde\gamma_i^j$ (with $i, j \leq 0$) in the SW cone, see Section \ref{sec:new-half-int},  would be $\gamma_{i-1}^{-i-j}$ or $\sigma\gamma_{i-1}^{-i-j}$ in this fashion.  
\end{remark}

\begin{remark}
\label{rmk:rho} Let $\rho$ be the permutation $(03)(12)$. Then $\rho$ has a nice action on the brane tiling as a reflection over a line of slope $-\frac{\sqrt{3}}{3}$, and in terms of the hexagonal notation, this is exhibited in Figure \ref{fig: rho}.  Additionally, we have the symmetry $\gamma_{j}^i = \rho \gamma_{i}^j$ comparing the two under hexagonal notation.  In particular, the even alcoves on the overlap of regions I and II are fixed under $\rho$ and correspond to $\gamma_{i}^i$.  Note also that the permutation $\alpha$ sending region I to II and VII to VIII is $\rho\sigma = \sigma \rho$. 
\end{remark} 

\begin{figure}[H]
    \centering
    \includegraphics[keepaspectratio=true, width=80mm]{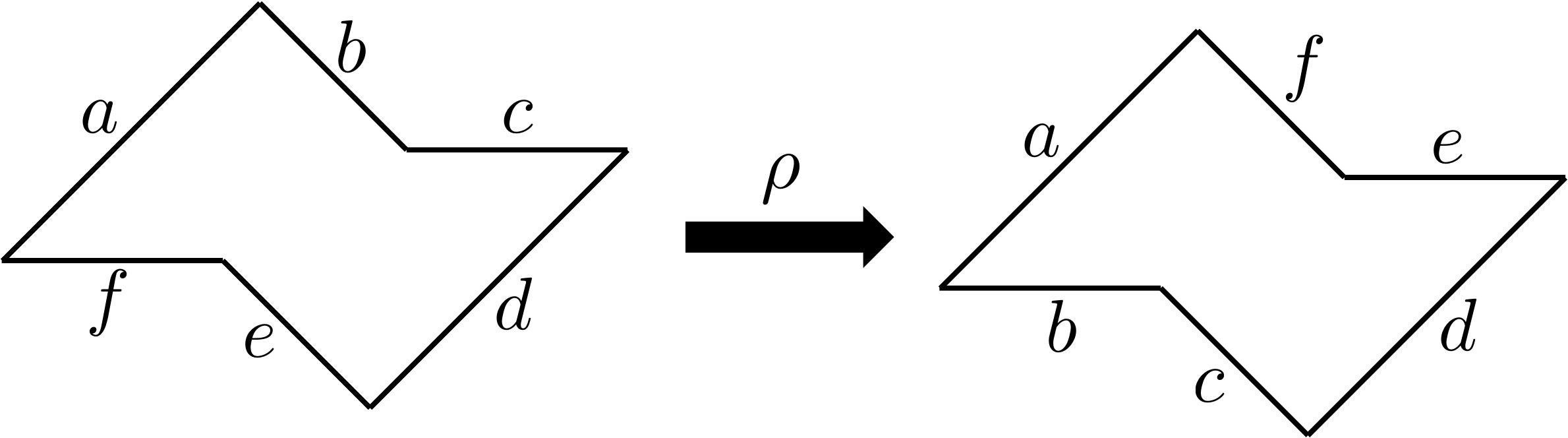}
    \caption{\small The action of $\rho = (03)(12)$ on hexagons.}
    \label{fig: rho}
   
\end{figure}

As a last case, note that $\gamma^{0}_n = (n-1,0, n-1, n-1, -1, n)$ and these are exactly the rotated Aztec dragons, $\rho(\sigma \gamma_0^{n})$ obtained by mutating along $213213\ldots$ an even number of steps in region II.

\end{subsection}

\begin{subsection}{Examples of Northeast Aztec Castles}

We begin with examples of Aztec integer dragons. The last two cluster variables corresponding to alcove $(0,1)$ are associated to the Aztec castles $D_1 = \sigma \gamma^{1}_{0} = (0,0,0,1,1-1)$ and $\gamma^{1}_0 = (1,1,-1,0,0,0)$, as in Figure \ref{fig: sigma D1}.  Then, we also have the two cluster variables corresponding 
to alcove $(1,1)$, i.e. $\gamma^{0}_{1} = (1,0,0,0,-1,1) = \rho\sigma\gamma^1_0$ and $\sigma\gamma^{0}_{1} = (0,-1,1,1,0,0) = \rho\gamma^1_0$ in Figure \ref{fig: sigmarho D1}.

\begin{figure}[H]
    \centering
    \includegraphics[keepaspectratio=true, width=50mm]{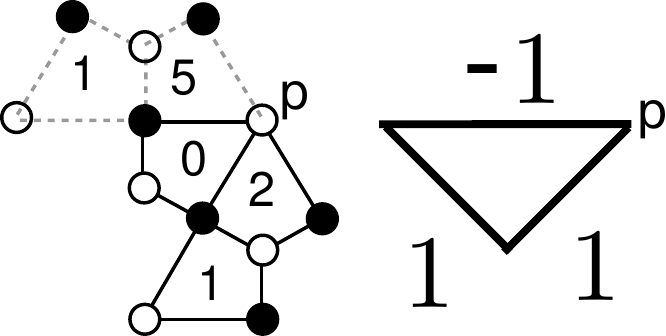} \hspace{4em}
    \includegraphics[keepaspectratio=true, width=50mm]{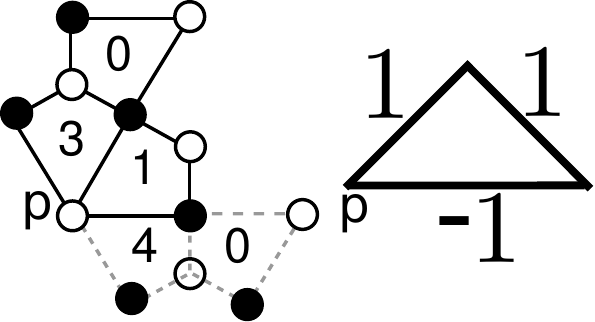}
    \caption{\small The castles $\sigma\gamma^{1}_{0}$ (also known as the Aztec dragon $D_{1}$) and $\gamma_0^{1}$.}
     \label{fig: sigma D1}
\end{figure} 

\begin{figure}[H]
    \centering
    \includegraphics[keepaspectratio=true, width=50mm]{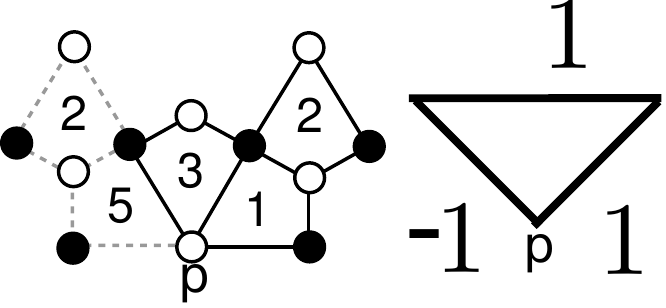} \hspace{4em}
    \includegraphics[keepaspectratio=true, width=50mm]{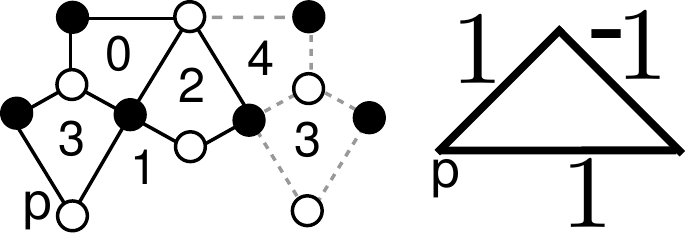}
    \caption{\small The castles $\sigma\gamma_1^{0} = \rho\gamma_0^1$ and $\gamma_1^{0} = \rho\sigma\gamma_0^1$, which are a rotation and reflection of $D_1$.  These correspond to even alcoves in region II.}
     \label{fig: sigmarho D1}
\end{figure} 

\noindent We also show $D_2 = \sigma \gamma^{2}_0 = (1,1,0,2,2,-1)$, which corresponds to the penultimate cluster variable for the even alcove $(0,2)$, with its associated hexagonal notation in Figure \ref{fig: sigma D2}.

\begin{figure}[H]
    \centering
    \includegraphics[keepaspectratio=true, width=90mm]{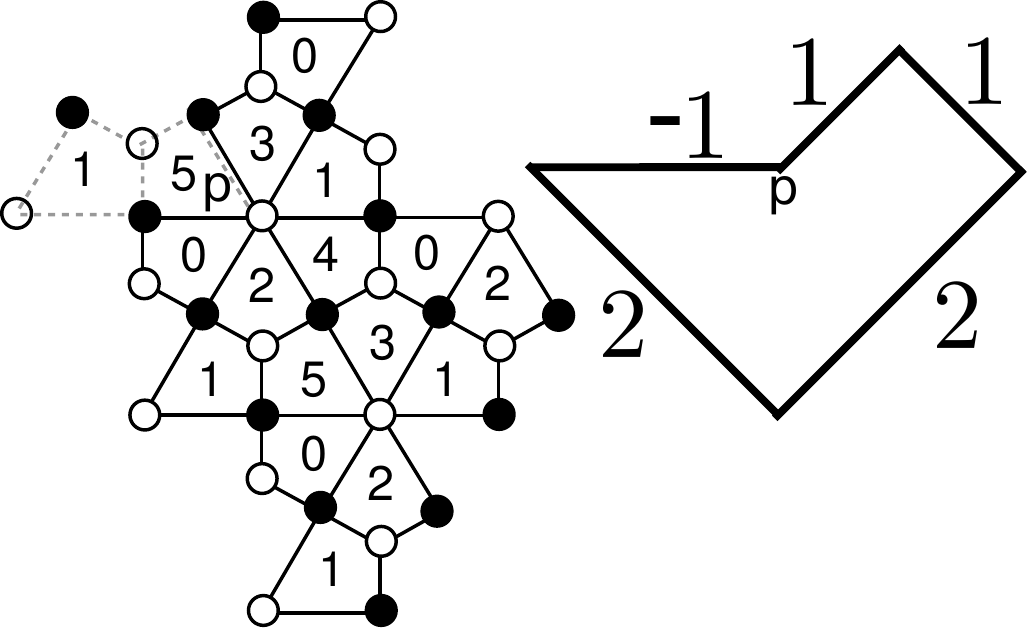}
    \caption{\small The castle $\sigma\gamma^{2}_{0}$, also known as the Aztec dragon $D_{2}$.}
     \label{fig: sigma D2}
\end{figure} 

Next, we also illustrate the Aztec castles $\gamma_{1}^1$ and $\gamma_{3}^3$ in Figures \ref{fig:NewEg} and \ref{fig:hexex}.

\begin{figure}[H]
    \centering
    \includegraphics[keepaspectratio=true, width=100mm]{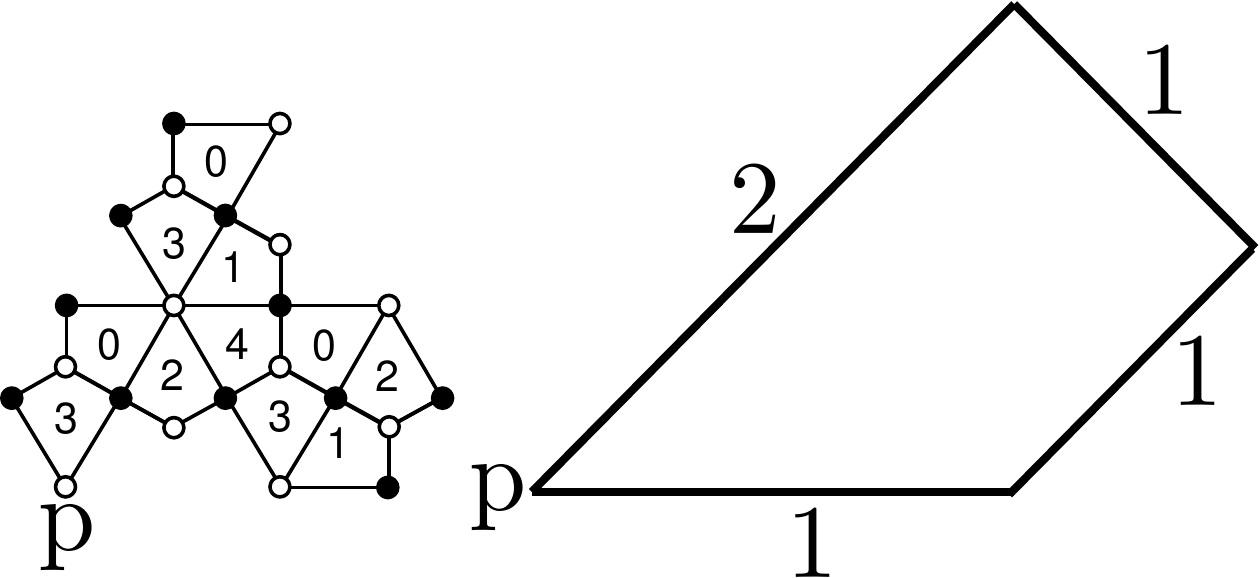}
    \caption{\small The castle $\gamma^{1}_{1}$ as a subgraph of the dP3 brane tiling with its corresponding hexagonal notation.  Notice that since this is on the intersection of regions I and II, it is symmetric with respect to $\rho$, which is a reflection about a line of slope $-\sqrt{3}/3$, i.e. rotation $\sigma$ applied to reflection about a line of slope $\sqrt{3}$.}
    \label{fig:NewEg}
\end{figure}

\begin{figure}[H]
    \centering
    \includegraphics[keepaspectratio=true, width=80mm]{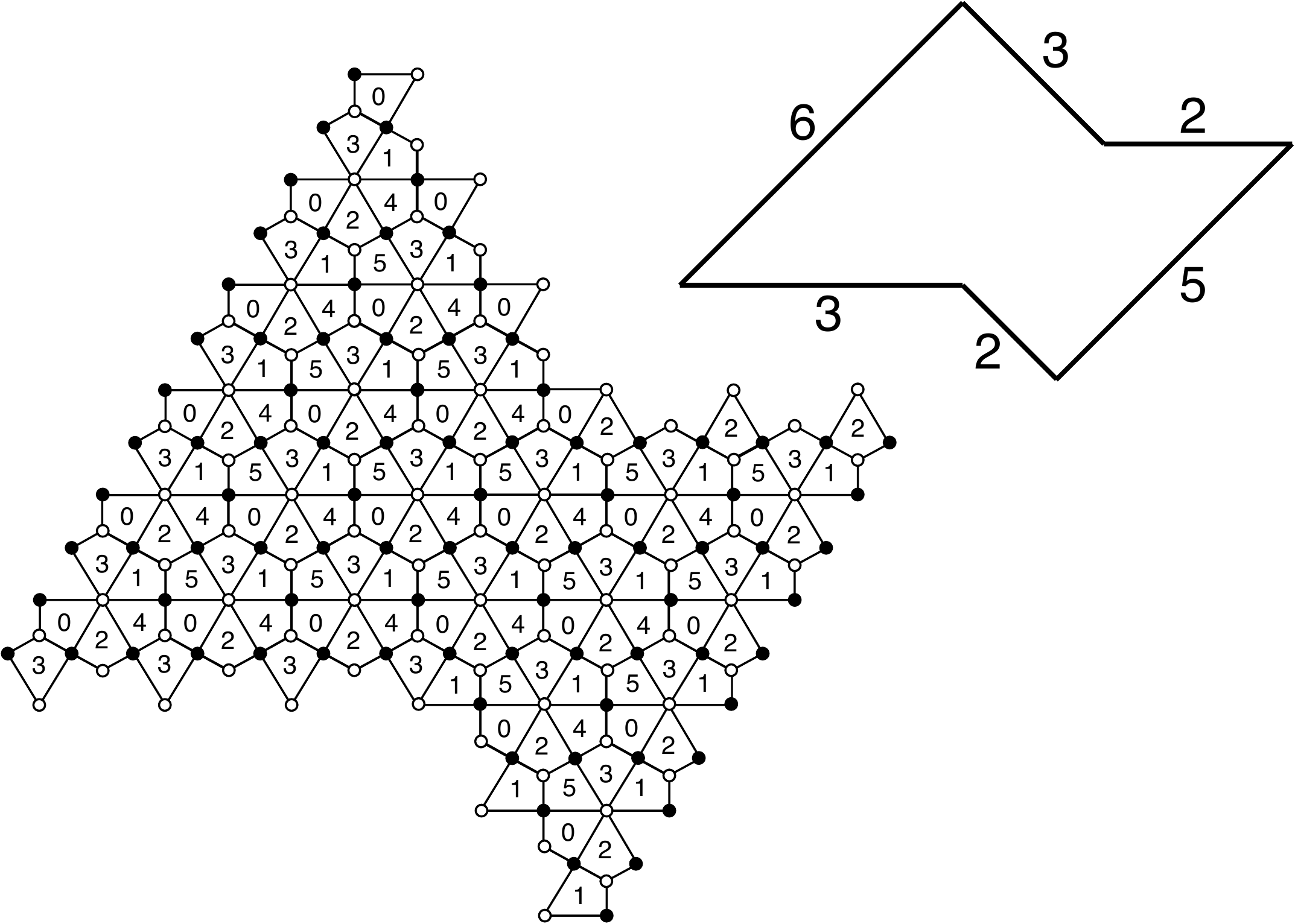}
    \caption{\small The castle $\gamma^{3}_{3}$ as a subgraph of the dP3 brane tiling with its corresponding hexagonal notation.  Note that 
the even alcove $(3,3)$ would correspond instead to $\sigma\gamma^3_3$ since $3\equiv 0 \mod 3$.}
    \label{fig:hexex}
\end{figure}

\end{subsection}

\begin{subsection}{The Northeast Cone Theorem}
\label{sec:intcone}

Now we state the main result of this section in full detail. Theorem \ref{thm:intcone} gives the correspondence between alcoves of the NE cone (region I) and NE Aztec castles. 

\begin{theorem}
\label{thm:intcone}
(1) For $0 \leq i \leq j$, let $(i,j)$ label an even alcove in the NE cone (using the coordinates from Section \ref{sec:introduction}).  Let $\tau_{a_1}\tau_{a_2}\cdots \tau_{a_{2i+2j}}$ denote 
the canonical path from the origin to this alcove (which takes $2j$ diagonal steps followed by $2i$ horizontal steps).  The mutation $\tau_{a_{2i+2j}} = \mu_{2k-2}\circ \mu_{2k-1}$ for $a_{2i+2j}=k \in \{1,2,3\}$.  Let 
$y_{i,j}$ and $y_{i,j}'$ denote the penultimate and final cluster variables reached by this mutation sequence, respectively.

Then for $(i,j) \not = (0,0)$, $y_{i,j} = 
\begin{cases}  
c(\sigma\gamma_i^j) &\mathrm{~if~}i \equiv 0 \mod 3  \\ c(\gamma_i^j) &\mathrm{~otherwise}\end{cases}$, and 

$y_{i,j}' = 
\begin{cases} 
c(\gamma_i^j) &\mathrm{~if~}i \equiv 0 \mod 3  \\ c(\sigma\gamma_i^j) &\mathrm{~otherwise}\end{cases}$.

(2) For $1 \leq i \leq j$, let $\{i,j-1\}$ label an odd alcove in the NE cone.  Let $\tau_{a_1}\tau_{a_2}\cdots \tau_{a_{2i+2j-1}}$ denote 
the canonical path from the origin to this alcove (which takes $2j$ diagonal steps followed by $2i-1$ horizontal steps).  This time, let  
$z_{i,j-1}$ and $z_{i,j-1}'$ denote the penultimate and final cluster variables reached by this mutation sequence, respectively. 

Then $z_{i,j-1} = 
\begin{cases} c(\gamma_i^{j-1}) &\mathrm{~if~}i \equiv 2 \mod 3, \mathrm{~ and} \\ c(\sigma\gamma_i^{j-1}) &\mathrm{~otherwise}\end{cases}$, and

$z_{i,j-1}' = 
\begin{cases} c(\sigma\gamma_i^{j-1}) &\mathrm{~if~}i \equiv 2 \mod 3, \mathrm{~ and} \\ c(\gamma_i^{j-1}) &\mathrm{~otherwise}\end{cases}$.
\end{theorem}

\end{subsection} 

\begin{subsection}{Condensations on Northeast Aztec Castles}

The work of detailing recursions on the weights of NE Aztec castles will rely on a method known as graphical condensation which was first introduced by Kuo in \cite{kuo} and further generalized by Speyer \cite{speyer}. We will be utilizing Speyer's reformulation of the condensation theorem multiple times throughout this section as well as in Section \ref{sec:new-half-int}, so it is useful to provide a brief review of the necessary facts. 

\begin{subsubsection}{The Condensation Theorem}
The condensation theorem relates the weights of a graph $G$ to smaller sections of that same graph by dividing G into nine disjoint sets. This decomposition will form the basis of the recursions on the weights of Aztec castles.  \\ 

\noindent Partition the vertices of a planar bipartite graph $G$ into nine distinct sets:
\[
V(G) = C \sqcup N \sqcup S \sqcup W \sqcup E \sqcup NE \sqcup NW \sqcup SE \sqcup SW
\]
\noindent and assume the following conditions hold:
\begin{enumerate}
\item  \label{cond1} The edge connections between the nine sets agree with those illustrated in Figure \ref{fig:condthm}. Thus, there cannot be any edge connecting N, S, E, or W to C and other similar restrictions. 

\item \label{cond2} The boundary vertices must be black in the SW and NE sectors and white in the SE and NW sectors. 

\item \label{cond3} N, S, W, E, and C contain the same number of black and white vertices. In SW and NE, the number of black vertices is one more than the number of white vertices and in SE and NW the number of white vertices is one more than the number of black vertices.
\end{enumerate}

\noindent Then:
\[
\begin{split}
w(G)w(C) &= w(N \cup NE \cup NW \cup C)w(S \cup SE \cup SW \cup C)w(E)w(W) \\ &+ w(E \cup NE \cup SE \cup C)w(W \cup NW \cup SW \cup C)w(N)w(S)
\end{split}
\]
\begin{figure}[h]
    \centering
    \includegraphics[keepaspectratio=true, width=100mm]{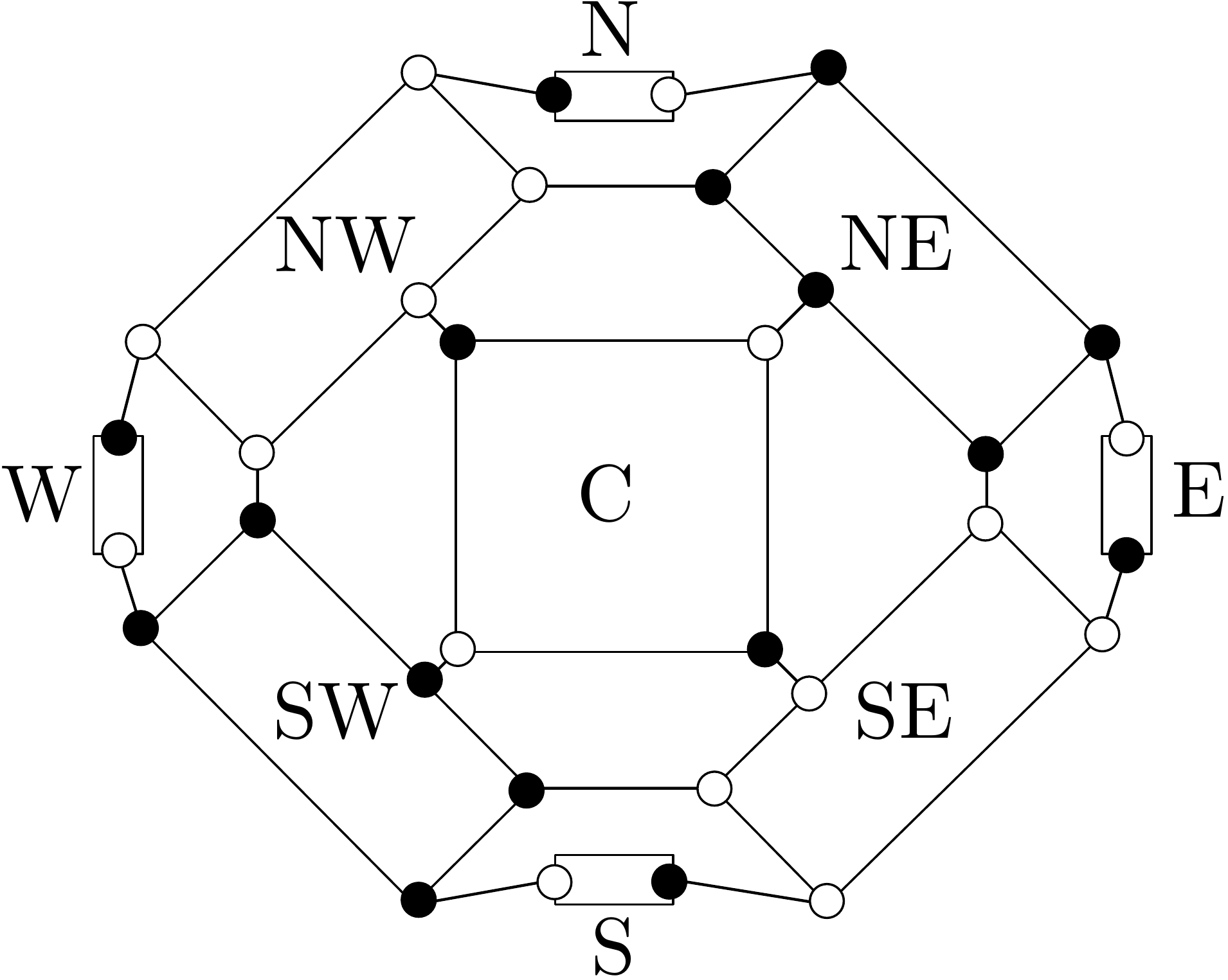}
	\caption{\small The possible allowed edge connections between the nine sets of vertices in a successful application of the condensation theorem. Boundary vertices must be white for the NW and SE sectors and black for the NE and SW sectors.  }
	\label{fig:condthm}
\end{figure}

\end{subsubsection}

\begin{subsubsection}{General Weight Relation}
\label{sec:genweight}

\begin{lemma}
\label{lem:genweight}
For $1 \leq i \leq j$, the following relation holds.
\begin{equation*}
\begin{split}
w(\gamma^{j}_i) w(\sigma \, \gamma^{j - 1}_{i-1}) = 
(w(\sigma \, \gamma^{j}_{i - 1}) w(\gamma^{j-1}_i) + w(\gamma^{j }_{i-1}) w(\sigma \, \gamma^{j-1}_{i})) \left( \frac{1}{x_0 x_1 x_2 x_3 x_4 x_5}\right). 
\end{split}
\end{equation*}

\end{lemma}

\begin{proof}

The case $(i,j) = (1,1)$ is a base case that is verified separately using a diagram in Section \ref{sec:proofintcone}, see Figure \ref{fig: partition_firstwedge}.

Suppose $j \geq 2$.  We define a Kuo condensation on $\gamma^{j}_i(p) = (i+j, j, i-1, i+j-1, j-1, i)_p$ as follows. Recall that $V(G)$ denotes the set of vertices of a graph $G$. Let $p_0$ be the endpoint of the path $P_1^1P_3^1$ beginning at the point $p$, let $p_1$ be the endpoint of the path $P_1^1$ beginning at the point $p$, and let $p_2$ be the endpoint of the path $P_3^1$ beginning at the point $p$. 

\begin{itemize}
\item $C = V(\sigma \gamma^{j - 1}_{i-1}(p_0)) = V(i+j-3,j-2,i-1,i+j-2,j-1,i-2)_{p_0}$.
\item $SW = V((i+j-2,j-2,i-1,i+j-2,j-2,i-1)_{p_2}) - V(C)$
\item $NW = V((i+j-2,j-1,i-1,i+j-2,j-1,i-1)_{p_1} - V(C)$
\item $NE = V((i+j-2,j-1,i-2,i+j-2,j-1,i-2)_{p_0}) - V(C)$
\item $SE = V((i+j-3, j-2,i,i+j-1,j,i-2)_{p_0}) - V(C)$
\item The induced subgraph on the remaining vertices has exactly two connected components. The vertices of the northernmost component are labeled $N$ and the vertices of the remaining component are labeled $W$. 
\item Regions $E$ and $S$ are empty.
\end{itemize}

See Figure \ref{fig:int_cone_CM} for a schematic of this decomposition.  Note that since $1\leq i \leq j$ and $j \geq 2$, each of these $6$-tuples either contains exclusively nonnegative parameters or at worst a $-1$ in the last coordinate of $C$ or the third and sixth coordinates of $NE$, which we use to trace out a subpath backwards.  
See Figure \ref{fig:int_subcase_CM} for this case.  We observe that

\begin{itemize}

\item $Q_1 := S \cup SW \cup SE \cup C = (i+j - 2, j-2, i, i+j-1, j-1, i-1) = \sigma \gamma^{j - 1}_{i}$  
\item $Q_2 := W \cup SW \cup NW \cup C = \gamma^{j - 1}_{i}$
\item $Q_3:= N \cup NW \cup NE \cup C = (i+j-1, j, i-2, i+j-2, j-1, i-1) = \gamma^{j}_{i-1}$
\item $Q_4 := E \cup NE \cup SE \cup C = \sigma \gamma^{j}_{i-1}$

\end{itemize}

As above, $Q_1$, $Q_2$, $Q_3$, and $Q_4$ contain exclusively nonnegative parameters, except for the cases of $Q_3$ and $Q_4$, which contain $-1$'s in the third (resp. sixth coordinate) when $i=1$.  However, in such cases, $Q_3$ and $Q_4$ are simply integer-order Aztec dragons, which agree with their coordinates as $\gamma_0^j$ or $\sigma \gamma_0^j$.

Thus the lemma follows if we can verify that the requirements of the condensation theorem are satisfied.  We provide an image of the general shape of the condensation, but leave the straightforward but tedious details verifying conditions (\ref{cond1})-(\ref{cond3}) to the reader.

\begin{figure}[H]
    \centering
    \includegraphics[keepaspectratio=true, width=100 mm]{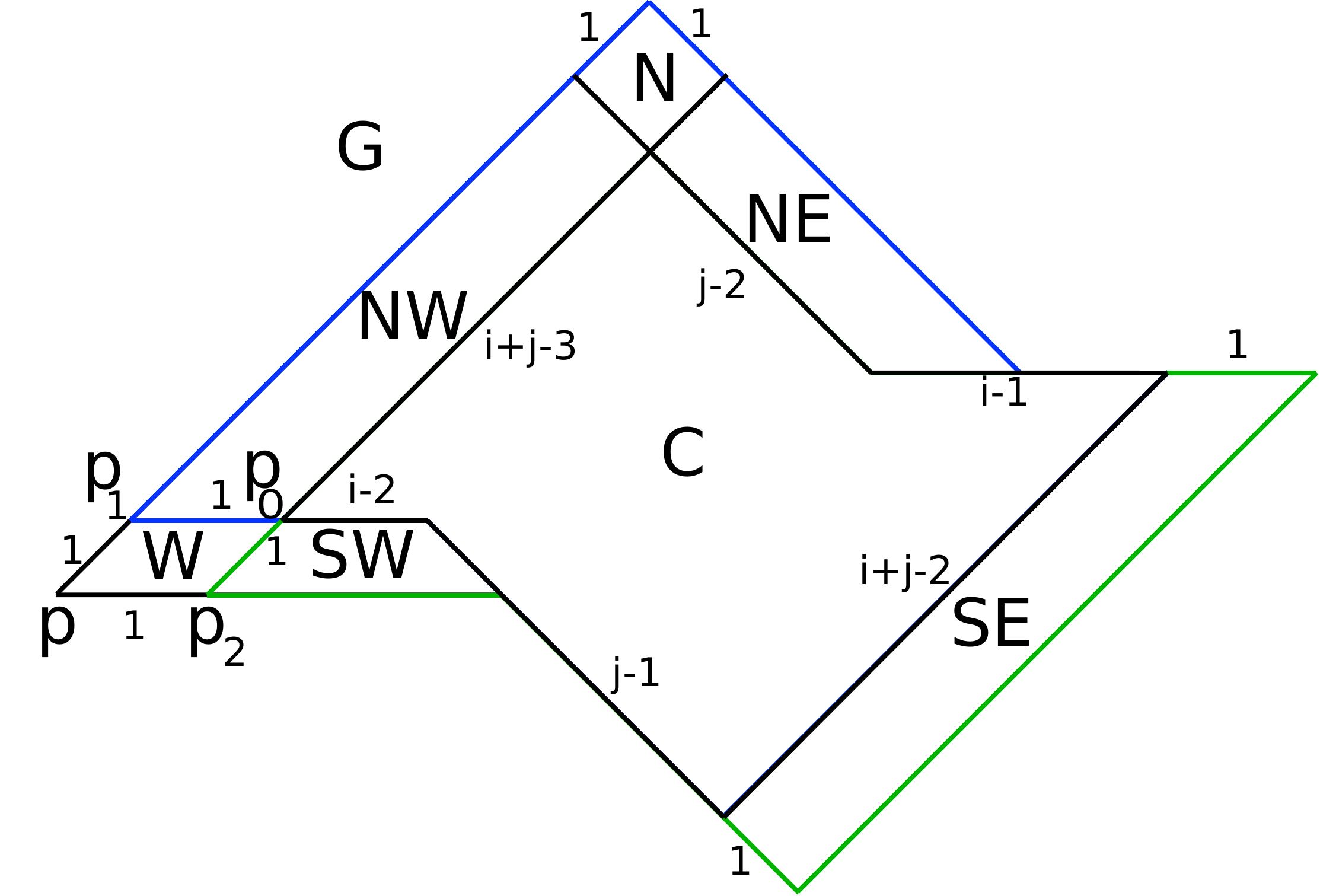}
    \caption{\small Schematic of condensation for $G = \gamma_{i}^j$ for $i, j \geq 2$.}
    \label{fig:int_cone_CM}
\end{figure}

\begin{figure}[H]
    \centering
    \includegraphics[keepaspectratio=true, width=80 mm]{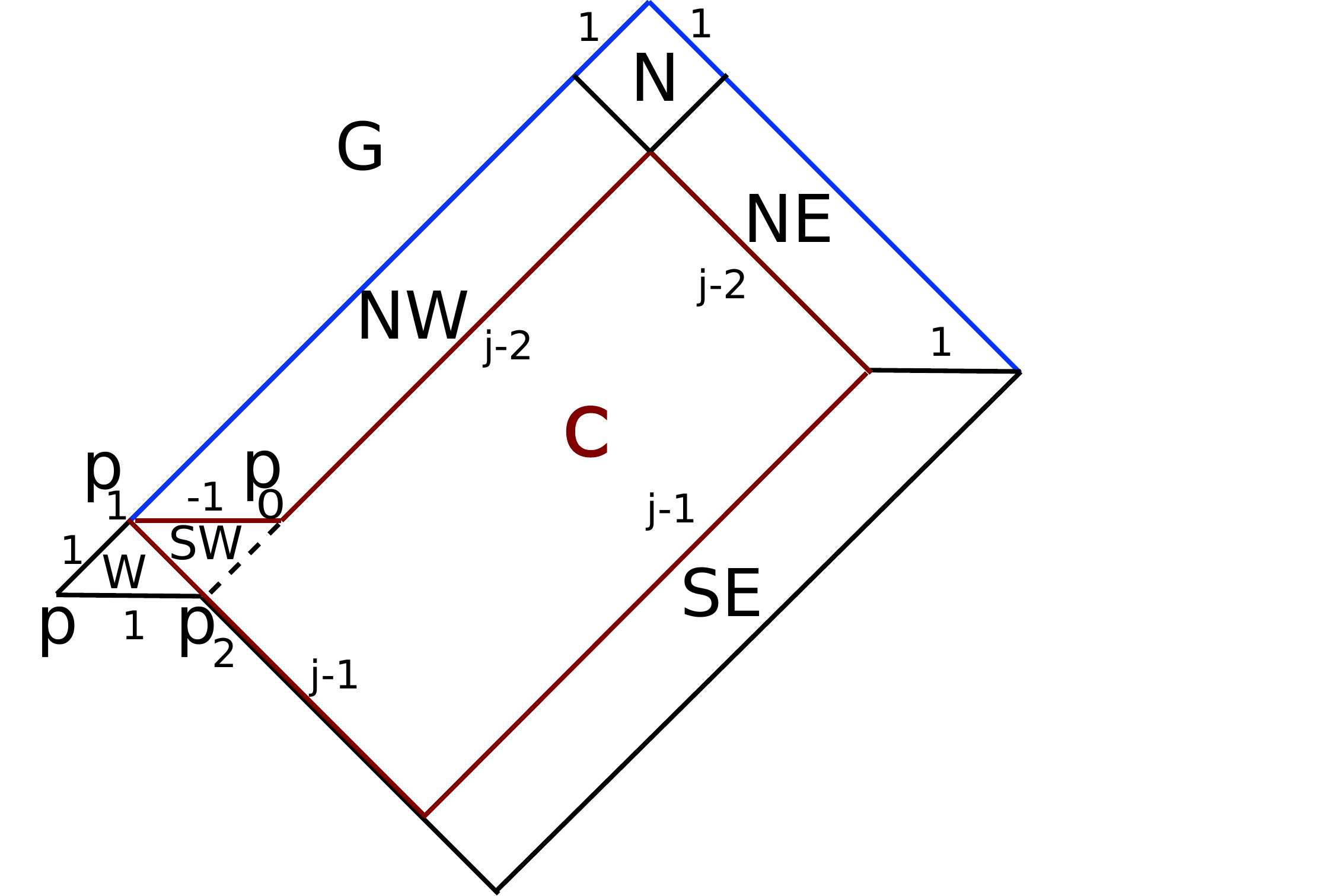}
    \caption{\small For the case $G = \gamma_{1}^j$, for $j\geq 2$, the same scheme essentially works, but notice that 
we will be going backwards along a subpath labeled by $-1$.  In particular, schematically, region $SW$ now 
looks like it is inside $C$, but that is an artifact of our notation.  In particular, the subgraphs $Q_1$, $Q_2$, $Q_3$, and $Q_4$ still fit in 
as desired.}
    \label{fig:int_subcase_CM}
\end{figure}

An example of this condensation is given in Figure \ref{fig:firstdecomp}.

\begin{figure}[H]
    \centering
    \includegraphics[keepaspectratio=true, width=100 mm, angle=270]{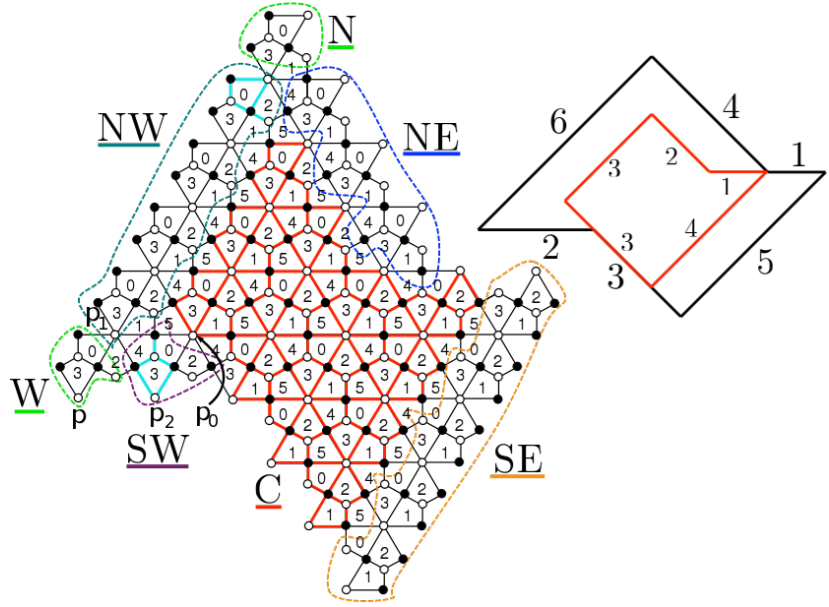}
    \caption{\small Example of condensation for the case $G = \gamma_{2}^4$ with $C = \sigma\gamma_{1}^3$.}
    \label{fig:firstdecomp}
\end{figure}

\end{proof}

\end{subsubsection}
\end{subsection}

\begin{subsection}{Covering Monomial Relation}
\label{subsec:int cov mon}

The following lemma will be combined with the weight relation to prove Theorem \ref{thm:intcone}.

\begin{lemma}
\label{lem:int cov mon}
For all $1 \leq i \leq j$:
 \[
\begin{split}
m (\gamma^j_{i})m (\sigma\gamma^{j-1} _ {i-1}) &= (x_ 0x_{1}x_{2}x_{3}x_{4}x_ 5)m (\sigma\gamma^{j-1} _ {i}) m (\gamma^{j}_ {i-1}) \\  &= (x_ 0x_{1}x_{2}x_{3}x_{4}x_ 5) m(\gamma^{j-1} _ {i})m (\sigma\gamma^j_ {i - 1}) 
\end{split}
\]
\end{lemma}

 \begin{proof}
The theorem can be proven by hand for $i,j = 1$. Hence assume $j \geq 2$ for the remainder of the proof.  Let $G$ be a subgraph of the brane tiling and let $m'(G)$ denote the set of quadrilaterals of the brane tiling included in and bordering $G$ (we represent the quadrilateral marked $k$ by $x_k$). Then we can rephrase Lemma \ref{lem:int cov mon} in this new notation as follows.
 \[ 
\begin{split}
m'(\gamma^j_i)\sqcup m'(\sigma\gamma^{j-1}_{i-1}) &= \{x_ 0, x_{1}, x_{2}, x_{3}, x_{4}, x_ 5\} \sqcup m'(\sigma\gamma^{j-1}_{i})\sqcup m'(\gamma^j_{i-1}) \\ &= \{x_ 0, x_{1}, x_{2}, x_{3}, x_{4}, x_ 5\} \sqcup  m'(\gamma^{j-1} _ {i}) \sqcup m'(\sigma\gamma^j_ {i - 1}) 
\end{split} 
\]

 Thus if for any $i,j$ we can overlap $m'(\sigma\gamma^{j-1}_{i})$ and $m'(\gamma^j_{i-1})$ such that
 $$
[m'(\sigma\gamma^{j-1}_{i})  \cup  m'(\gamma^j_{i-1})] \sqcup  \{x_ 0, x_{1}, x_{2}, x_{3}, x_{4}, x_ 5\} =  m'(\gamma^j_i)$$ 
and 
$$
m'(\sigma\gamma^{j-1}_{i})  \cap   m'(\gamma^j_{i-1}) = m'(\sigma\gamma^{j-1}_{i-1}),$$

\noindent then the lemma follows. 

The following argument is aided by the visuals provided in Figure \ref{fig:int_cone_CM} for the case when all side lengths are positive ($i,j \geq 2$) and Figure \ref{fig:int_subcase_CM} for the case when $i = 1$. 

Recall the condensation presented in Lemma \ref{lem:genweight}. We know $Q_3 = C \cup NW \cup N \cup NE = \gamma^{j}_{i-1}$ and $Q_1 = C \cup SW \cup SE = \sigma \gamma^{j - 1}_{i}$. Since the condensation partitions all vertices of $G = \gamma^{j}_i$, we also have $Q_1 \cap Q_3 = C = \sigma \gamma^{j - 1}_{i - 1}$. Verifying that the border squares of $Q_1$ and $Q_3$ only intersect at border squares of $C$, this implies $m'(Q_1) \cap m'(Q_3) = m'(C)$, which is precisely the second required condition above. Furthermore, $G - [Q_1 \cup Q_3] = W$, $m'(W)$ is disjoint from $m'(Q_1 \cup Q_3)$, and $m'(W) = \{ x_0, x_1, x_2, x_3, x_4, x_5\}$. This implies that $m'(\gamma^{j}_i) = m'(W) \sqcup (m'(Q_1) \cup m'(Q_3))$, which was the first required condition. We conclude that $m (\gamma^j_{i})m (\sigma\gamma^{j - 1} _ {i-1}) = (x_ 0x_{1}x_{2}x_{3}x_{4}x_ 5)m (\sigma\gamma^{j - 1} _ {i}) m (\gamma^j_ {i - 1})$. 

Since $N$ has the same shape as $W$, we apply the same reasoning using $Q_2 = W \cup SW \cup NW \cup C = \gamma^{j-1}_{i}$ and $Q_4 = E \cup NE \cup SE \cup C = \sigma \gamma^{j}_{i-1}$ to conclude $m (\gamma^j_{i})m (\sigma\gamma^{j - 1} _ {i-1}) = (x_ 0x_{1}x_{2}x_{3}x_{4}x_ 5) m(\gamma^{j - 1} _ {i})m (\sigma\gamma^j_ {i - 1})$. 

\end{proof}

\end{subsection}

\begin{subsection}{Proof of the Northeast Cone Theorem}
\label{sec:proofintcone}

Now we combine Lemmas \ref{lem:genweight}, \ref{lem:int cov mon}, and a final lemma handling odd alcoves to prove the main result of this section, Theorem \ref{thm:intcone}.

\begin{lemma}
\label{lem:oddsigma}

At the level of ordered sets, $\{y_{i,j}, y'_{i,j}\} = \{z'_{i,j},z_{i,j}\}$ if $i \equiv 1 \mod 3$ and $\{y_{i,j}, y'_{i,j}\} = \{z_{i,j},z'_{i,j}\}$ otherwise. The same 
condition holds if we mutate along canonical paths when $i>0$.

\end{lemma}

\begin{proof}

We show this for each row of the NE cone. Suppose that $j \equiv 2 \mod 3$. The remaining cases follow with the same strategy applied. We first show this explicitly for the odd alcove $\{1,j-1\}$. We represent the cluster at the even alcove $(i,j)$ (resp. odd alcove $\{i,j\}$) by $\mathcal{C}_{(i,j)} = \{ c_0, c_2, c_4 \}$ (resp. $\mathcal{C}_{\{i,j\}} = \{ c_0, c_2, c_4 \}$), where $c_i$ denotes the set of cluster variable at vertex $v_i$. Note that the path to the even alcove $(1, j - 1)$ is given by $\beta^{\frac{2}{3}(j-2)}12|13$ and the path to the odd alcove $\{1,j-1\}$ is given by $\beta^{\frac{2}{3}(j-2) + 1} 1|3$. Thus we have 
\[
\begin{split}
\mathcal{C}_{\{1,j - 1\}} &= \{ y_{0,j}, y_{0,j-1}, z_{1,j-1}  \}; \\
\mathcal{C}_{(1, j - 1)} &= \{ z_{1,j-2} , y_{0,j - 1} , y_{1,j-1} \} 
\end{split}
\] 
By the relations among the $\tau_j$'s, any path to the same alcove will yield the same (unordered) cluster. From the construction of the lattice, $\tau_1 \, \mathcal{C}_{(1, j - 1)} = \mathcal{C}_{\{1,j-1\}}$. Because the operation $\tau_1$ changes only the first entry of the cluster in the given notation, we can infer that $z_{1,j-1}$ and $y_{1,j-1}$ are the same up to $\sigma$. By the factorization phenomenon, $z_{1,j-1}$ and $y_{1,j-1}$ have the same exponents. Hence, they differ at most by their leading term, which is determined by the parity of the number of $3$'s in either path. It is clear that the path to the even alcove has precisely one more $3$ than the path to the odd alcove, so indeed $z_{1,j-1} = \sigma y_{1,j-1}$.  

The statement for larger $i$ follows in the same way by keeping track of the parity of the number of $\tau_i$'s in the two different canonical paths. In summary, the canonical path to the even alcove $(i,j-1)$ and the canonical path to the odd alcove $\{i,j-1\}$ differ as words by a single $3$ if $j \equiv 2 \mod 3$. Hence, only the $z_{i,j-1}$ produced after applications of $\tau_3$ will differ by $\sigma$ from $y_{i,j-1}$. And $\tau_3$ is the final $\tau$ applied to reach an odd alcove on the row $j \equiv 2 \mod 3$ iff $i \equiv 1 \mod 3$.  

\end{proof}

\begin{proof}[Proof of Theorem \ref{thm:intcone}]

Note first that $c(\sigma \gamma^1_0) = y_{0,1}$, $c(\gamma^0_1) = z_{1,0}$, and $c(\gamma^1_1) = y_{1,1}$. The first two values can be checked explicitly and are also handled in \cite{zhang}. The final case follows from the decomposition presented in Figure \ref{fig: partition_firstwedge}.

In fact, note in this case $G = \gamma_{1}^{1}$, $C = \sigma\gamma_{0}^{0} = \emptyset$, and $Q_1 = \sigma\gamma_{1}^0$,
$Q_2 = \gamma_{1}^0$, $Q_3 = \gamma_{0}^1$, and $Q_4 = \sigma\gamma_{0}^1$, up to including an extra path of length two ($2$ edges) which dangle off of these subgraphs and do not effect the enumeration of perfect matchings.  See Figures \ref{fig: sigma D1} and \ref{fig: sigmarho D1}.

The covering monomial is easily computed by hand and compensates for the dangling edges. 

\begin{figure}[H]
    \centering
    \includegraphics[keepaspectratio=true, width=60mm]{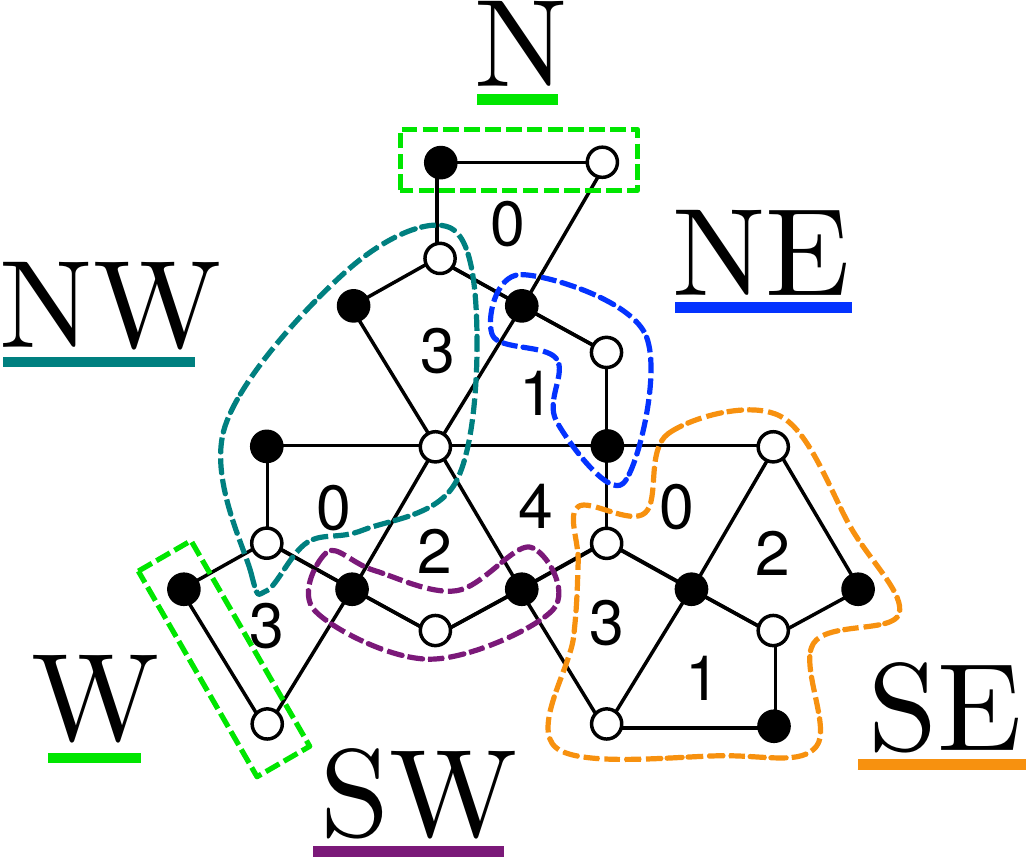}
    \caption{\small The condensation for the base case of $\gamma^{1}_{1}$ which corresponds to the alcove $(1,1)$.  }
    \label{fig: partition_firstwedge}
   
\end{figure}

Therefore, the theorem holds for the row $j = 1$ in the NE cone. Now let us consider rows higher than $j = 1$. A canonical path falls into one of the three forms described in Section \ref{subsec:cox}, based on the value of $j$ modulo $3$. By the work in \cite{zhang}, $y_{0,j} = D_{j} = \sigma \, \gamma^{j}_{0}$. 

Now we induct on the rows of the NE cone. Suppose that the theorem holds for row $j - 1$. We show that the theorem holds for the row $j$. This requires handling three cases, one for each type of canonical path. We handle the case $j \equiv 2 \mod 3$ explicitly with the two remaining cases following analogously. 

As can be checked, Lemma \ref{lem:oddsigma} combined with the inductive step verifies the statement of the theorem for the odd alcoves $\{1,j-1\}, \{2,j-1\}, \ldots, \{j-1, j-1\}$. The final odd alcove $\{j,j-1\}$ is a special case handled separately at the end. For the even alcoves, we will use the exchange relation and the labeling of the quiver at each step to verify that the exchange relation is compatible with the covering monomial and weight relations.

We will verify this first directly for the alcoves $(1,j), (2,j)$ and $(3,j)$, which serve as base cases. Before doing so, it will be convenient to introduce notation that represents the seed (the labeling of the quiver and the variables at each vertex) at any stage in the $\tau$-mutation sequence. The seed is represented by the following matrix.  
\[
\begin{pmatrix}
a_0 & a_1 & a_2 \\
X_0 & X_1 & X_2 \\
a_3 & a_4 & a_5 \\
X_3 & X_4 & X_5 \\
\end{pmatrix}
\]
Here, $\{a_0, \ldots, a_5 \}$ is a permutation of the set $\{0,1,2,3,4,5 \}$. Starting from the vertex labeled $0$ on the initial quiver and reading clockwise, $Q$ currently has the labeling $a_0-a_1-a_2-a_3-a_4-a_5$. The entry $X_i$ directly below $a_i$ specifies the cluster variable at the vertex labeled $a_i$. For example, 
$\begin{pmatrix}
0 & 5 & 3 \\
x_0 & x_5 & x_3 \\ 
1 & 4 & 2 \\
x_1 & x_4 & x_2 \\
\end{pmatrix}$ specifies the initial cluster. Observe also that because the action of any $\tau$ simply swaps antipodal vertices, $a_0$ and $a_3$ are either $0$ or $1$ after any $\tau$-mutation sequence. The analogous statements hold for $a_1$ and $a_4$ as well as $a_2$ and $a_5$. 

The path to the first alcove we consider, $(1,j)$, is given by $\beta^{\frac{2}{3}(j-2) + 1}1|32$. By the inductive hypothesis on $j$, the first part of this proof, the result of \cite{zhang} ($y_{0,j} = c(D_j)$), and the action of $\tau_i$ on the quiver, the seed has the following form immediately before applying the final $\tau$ in the given path.
\[
\begin{pmatrix}
0 & 5 & 2  \\
c(\sigma \, \gamma^{j}_{0}) & c(\gamma^{j-1}_{1}) & c(\sigma \, \gamma^{j-1}_{0}) \\
1 & 4 & 3 \\
c(\gamma^{j}_{0}) & c(\sigma \, \gamma^{j-1}_{1})  & c(\gamma^{j-1}_{0}) \\
\end{pmatrix}
\]
The exchange relation for the mutation $\mu_2$, the first mutation in $\tau_2$, is thus:
\begin{equation*}
x_{1,j} \cdot c(\sigma \, \gamma^{j-1}_{0}) = c(\gamma^{j}_{0}) c(\sigma \, \gamma^{j-1}_{1}) + c(\sigma \, \gamma^{j}_{0}) c(\gamma^{j-1}_{1})
\end{equation*}

If we multiply the general condensation relation by the covering monomial relation leaving $j$ as $j$ and setting $i = 1$, we recover the previous equation with $x_{1,j}$ replaced by $c(\gamma^{j}_{1})$. This implies $x_{1,j} = c(\gamma^{j}_{1})$ as desired. 

The same method verifies that $x_{2,j} = c(\gamma^{j}_{2})$ and $x_{3,j} = \sigma \, c(\gamma^{j}_{3})$. With this information, we can complete the inductive step. 

Suppose that our graphs are labeled according to the statement of the theorem has up to some odd alcove $\{i,j-1\}$ with $i \leq j$. We verify that the theorem holds for the alcove $(i, j)$. Then based on the value of $i$ modulo $3$, we have three cases, one of which we verify explicitly. 

Suppose $i \equiv 1 \mod 3$. Then the path to $(i, j)$ has the form $\beta^{\frac{2}{3}(j-2) + 1}1|(321)^{2u}32$, for some $u \in \mathbb{N}$. Observe that before the application of $\tau_2$, the quiver has the same vertex labeling as before entering the even alcove $(1,j)$. Then by the two inductive hypotheses the seed has the following form: 
\[
\begin{pmatrix}
0 & 5 & 2 \\
c(\sigma \, \gamma^{j}_{i - 1})  &  c(\gamma^{j - 1}_{i}) &c(\sigma \, \gamma^{j-1}_{i - 1}) \\
1 & 4 & 3 \\  
c(\gamma^{j}_{i - 1}) & c(\sigma \, \gamma^{j - 1}_{i})  & c(\gamma^{j-1}_{i - 1}) \\
\end{pmatrix}
\]
Then the exchange relation for $\mu_2$, the first step of $\tau_2$ is given below. 
\begin{equation*}
\begin{split}
x_{i, j} \cdot c(\sigma \, \gamma^{j-1}_{i - 1}) &= 
c(\sigma \, \gamma^{j}_{i-1}) c(\gamma^{j-1}_{i}) \\ &+ c(\gamma^{j}_{i-1}) c(\sigma \, \gamma^{j-1}_{i}) 
\end{split}
\end{equation*}

Multiplying the condensation relation by the covering monomial leaving $i$ as $i$ and $j$ as $j$, we get the previous equation with $x_{i,j}$ replaced by $c(\gamma^j_i)$, as desired. 

For $i \equiv 2 \mod 3$, the same process of keeping track of the form of the quiver reveals that $x_{i, j} = c(\gamma^{j}_{i})$. Likewise, if $i \equiv 0 \mod 3$, $x_{i, j} = c(\sigma \, \gamma^{j}_{i})$, as desired. 

We handle separately the final odd alcove $\{j,j-1\}$. The canonical path to this alcove is given by $P = \beta^{\frac{2}{3}(j-2) + 1}1|(321)^{\frac{2}{3}(j-2) + 1}$. Note that this alcove as well as $(j-1, j)$ are on the boundary of the NE cone traced by the path $12131213 \ldots$. Specifically, $(j-1, j)$ and $\{j, j-1\}$ are intersected by the last two steps of the path $(1213)^{j-1}121$. It follows from applying the recursions that naturally result from the proof of Proposition $\ref{thm:factor}$ that the cluster variables produced by the last two steps of $P$ are related by $\rho$ or $\sigma \rho$. Because the cluster is independent of path, the same must hold true for $y_{j-1, j}$ and $z_{j, j-1}$. Using an argument based on the parity of the number of $2$'s and $3$'s in the original path $P$ and recalling that $j \equiv 2 \mod 3$, we can directly observe that $z_{j,j-1} = \rho y_{j-1,j} = \rho c(\gamma^{j-1}_{j}) = c(\gamma^{j}_{j-1})$ by Remark \ref{rmk:rho}. This is consistent with the statement of the theorem, as desired.

This concludes the inductive step for the case $j \equiv 2 \mod 3$. The remaining cases $j \equiv 0, 1 \mod 3$ follow by the same strategy, with details left to the reader. This completes the proof.

\end{proof}

\end{subsection}

\end{section}

\begin{section}{Working in the Southwest Cone}
\label{sec:new-half-int}

Now we describe the graphs that appear in the \textbf{SW cone} first introduced in Section \ref{subsec:cox}. These graphs, which we call \textbf{SW Aztec castles}, are in perfect duality with the graphs in the NE cone, a fact which greatly simplifies our proofs. We begin by introducing hexagonal and six-tuple notation for the SW Aztec castles.

\begin{subsection}{Hexagonal and Six-tuple Notation}
\label{subsec: halfint hex not}

Given a six-tuple $[a,b,c,d,e,f]$, where we use square brackets to distinguish this from the previous notation, we associate the hexagon exhibited in Figure \ref{fig: halfsubpaths} below.  Such tuples defines graphs $G$ in the brane tiling, just like before.  Notice that the four-step subpaths and lengths of these subpaths are different than in the NE cone, and we use a basepoint on the right-hand-side rather than the left-hand-side.  
 
\begin{figure}[H]
    \centering
    \includegraphics[keepaspectratio=true, width=100mm]{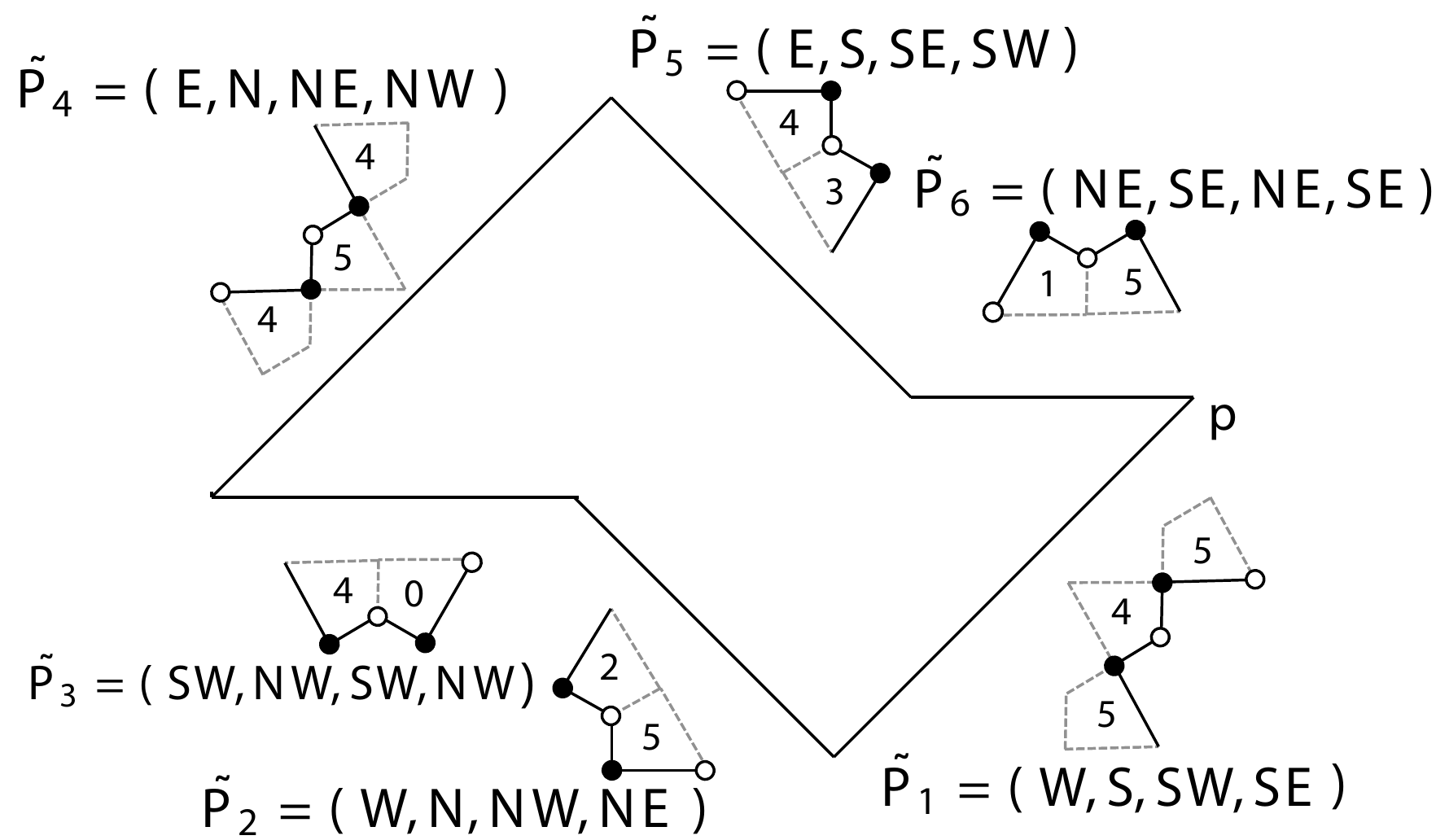}
    \caption{\small The six four-step subpaths, beginning at basepoint $p$, that trace out a boundary path $\tilde{P}$ placed around a SW Aztec castle to signify their position in the brane tiling. The subgraph induced by all of the vertices in $\tilde{P}$ is an Aztec castle in the SW cone.  }
    \label{fig: halfsubpaths}
\end{figure}

Associate the path 
\[
\tilde{P} = \tilde{P}_1^a \tilde{P}_2^{b+1} \tilde{P}_3^{c+1} \tilde{P}_4^{d} \tilde{P}_5^{e+1} \tilde{P}_6^{f+1} 
\]
which starts at the basepoint $p$ to the tuple $[a,b,c,d,e,f]$. Note that the transition of the path from $\tilde{P}_2$ to $\tilde{P}_3$ as well as $\tilde{P}_5$ to $\tilde{P}_6$ involves tracing out an edge twice. Both of these edges should be deleted from $\tilde{P}$, i.e. they are not included in the border of the SW Aztec castle. If $\tilde{P}$ is a simple closed curve after deletion, we define $G$ to be the subgraph induced by all vertices on the boundary path $\tilde{P}$ and in the interior region of $\tilde{P}$. 

As before, we can interpret negative exponents of these paths graphically for some special cases. Given a six-tuple $[a,b,c,d,e,f]$ with $c = -1$ or $f = -1$, construct the graph of $[a,b,c,d,e,f]$ as the interior of the path $\tilde{P}_1^a \tilde{P}_2^{b+1} \tilde{P}_3^{c+1} \tilde{P}_4^{d} \tilde{P}_5^{e+1} \tilde{P}_6^{f+1}$ with the leftmost square removed if $c = -1$ and the rightmost square removed if $f = -1$. 

\end{subsection}

\begin{subsection}{Southwest Aztec Castles and Duality}
\label{subsec:duality}
 
\begin{definition}
\label{def:half castles}
For $j<-1$ and $0 \geq i > j$, define the $\{i,j\}$-SW Aztec castle by:
\begin{equation*}
\tilde \gamma^j_i = [-i-j-1,-j-1,-i-1,-i-j-2,-j-2,-i].
\end{equation*}
These correspond to odd alcoves in the SW cone (up to $\sigma$).  As a base case, set $\tilde \gamma^{-1}_0$ to be a single quadrilateral in the brane tiling labeled $5$. 

\end{definition}

Below are two examples of SW Aztec castles. 

\begin{figure}[H]
\centering
\begin{subfigure}{.3\textwidth}
  \centering
  \includegraphics[keepaspectratio=true, width=60mm]{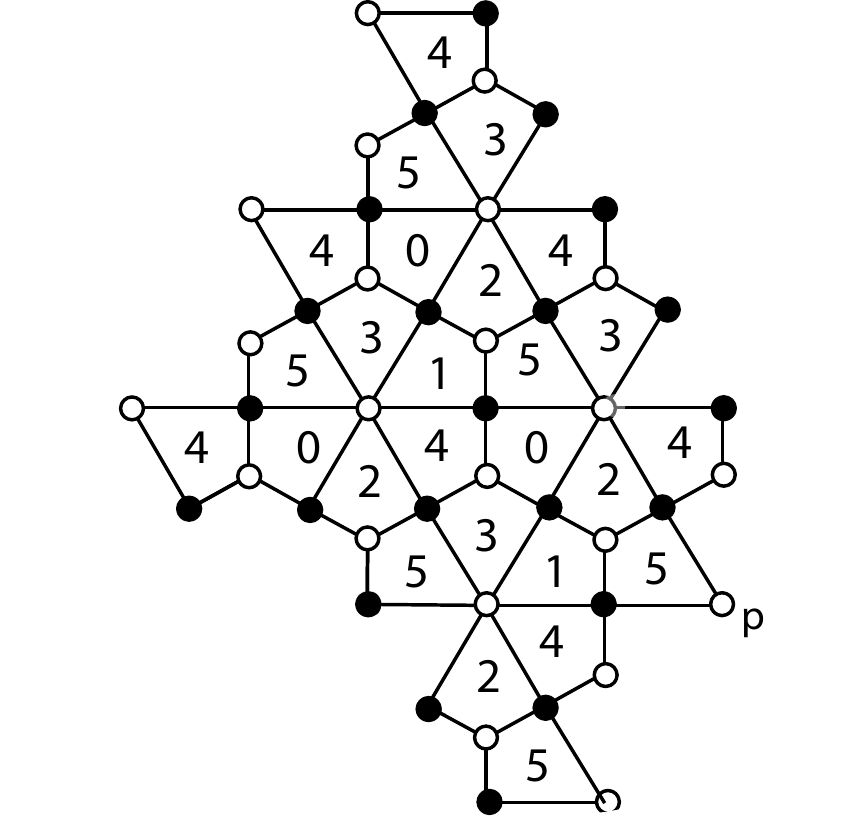}
  \label{fig:sub1}
\end{subfigure}%
\hspace{1 cm}%
\begin{subfigure}{.3\textwidth}
  \centering
  \includegraphics[keepaspectratio=true, width=100mm]{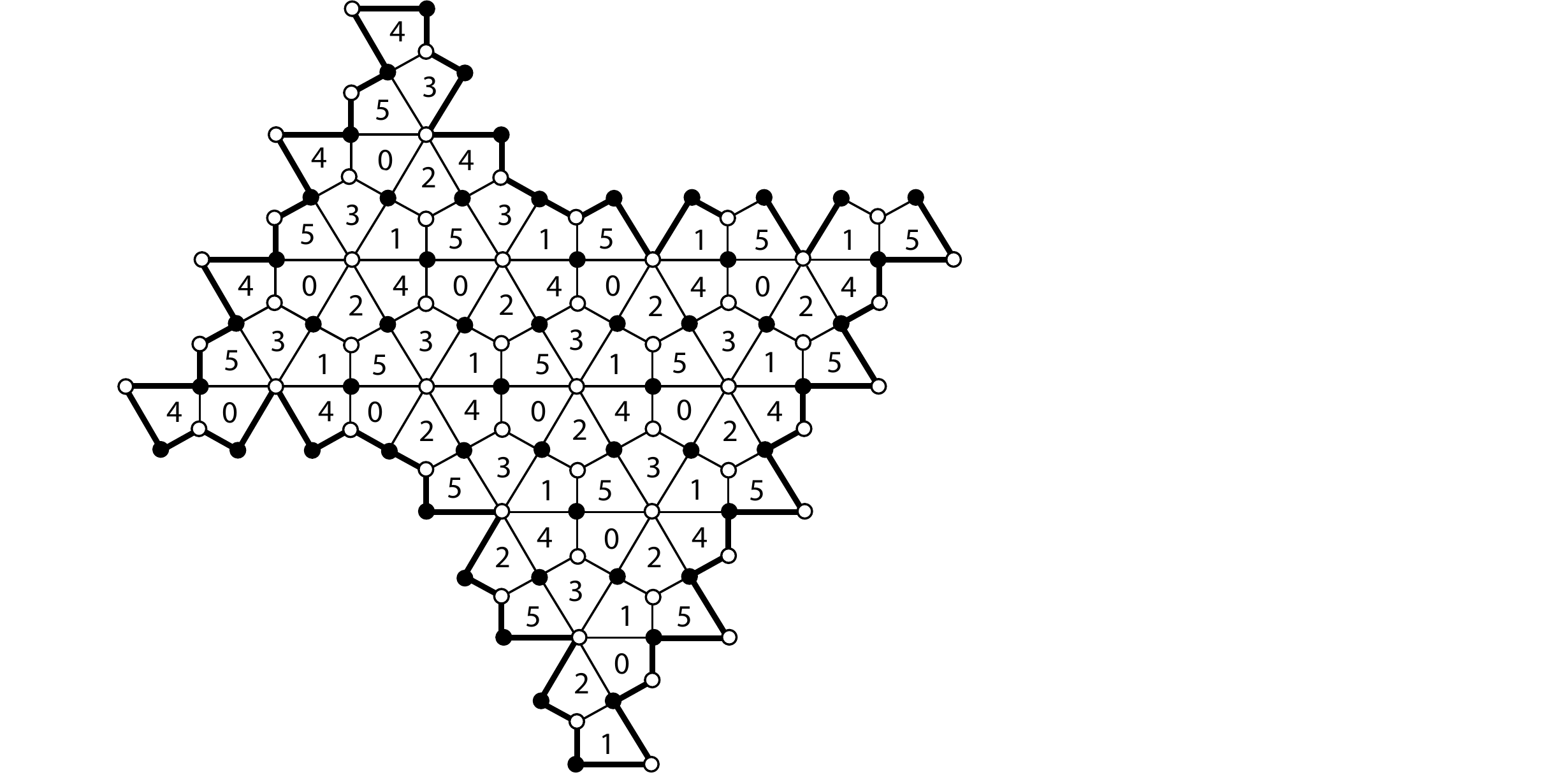}
  \label{fig:sub2}
\end{subfigure}
\caption{\textbf{(Left)}: A vertical reflection of the half-integer order Aztec dragon $D_{5/2}$, written $ \sigma \tilde{\gamma}^{-3}_0 = [1, 1, 0, 2, 2, -1]$ in hexagonal notation. \textbf{(Right)}: The SW Aztec castle $\tilde\gamma_{-2}^{-3} = [4,2,1,3,1,2]$. }
\label{fig:test}
\end{figure}

The permutation $\rho = (03)(12)$ has a nice behavior on these graphs as well. In general, we have $\tilde{\gamma}^{j}_{i} = \rho \tilde{\gamma}^{i - 1}_{j+1}$, mimicking the reflection between Region VII and VIII as in Table \ref{tab:perms}. 

Definitions \ref{def:integer castle} and \ref{def:half castles} imply that there is a natural map between NE and SW Aztec castles $\iota$ sending $(a,b,c,d,e,f)$ to $[a,b,c,d,e,f]$. Precisely, we have $\iota$ is an involution sending $\gamma_i^j \leftrightarrow \tilde{\gamma}_{-i}^{-j-1}$.

We will further show that $\iota$ has a natural action on (1) condensations involving Aztec castles, and (2) the covering monomial overlayings of Aztec castles. This leads to a brief proof of the main result of this section, that any cluster variable produced by canonical paths in the SW cone can expressed as $c(\tilde{\gamma}^j_i)$ or $c(\sigma \tilde{\gamma}^j_i)$.

\end{subsection}

\begin{subsubsection}{Weight Relations}

\begin{lemma}
\label{lem:genweight1}
For $0 \geq i > j$, the following relation holds:
\begin{equation*}
\begin{split}
w(\tilde{\gamma}^{j}_{i}) w(\sigma \, \tilde{\gamma}^{j + 1}_{i+1}) = (w(\sigma \, \tilde{\gamma}^{j + 1}_{i}) w(\tilde{\gamma}^{j}_{i + 1}) + w(\tilde{\gamma}^{j+1}_{i}) w(\sigma \, \tilde{\gamma}^{j}_{i + 1}))\left(\frac{1}{x_0 x_1 x_2 x_3 x_4 x_5}\right). 
\end{split}
\end{equation*}
\end{lemma}

\begin{proof}

The base case $i = -1$, $j = -2$ is handled separately.  It is similar to the $i=j=1$ case in the NE cone, see Figure \ref{fig: partition_firstwedge}, although this time, $G = \tilde\gamma_{-1}^{-2}$, the subgraph $C$ is $\sigma \tilde\gamma_{0}^{-1}$, i.e. a quadrilateral labeled $4$ in the brane tiling, and subgraphs $Q_1$, $Q_2$, $Q_3$, and $Q_4$ are rotated versions of $D_{3/2}$ up to vertical reflection, $\rho$, and $\sigma$.

\begin{figure}[H]
    \centering
    \includegraphics[keepaspectratio=true, width=70mm]{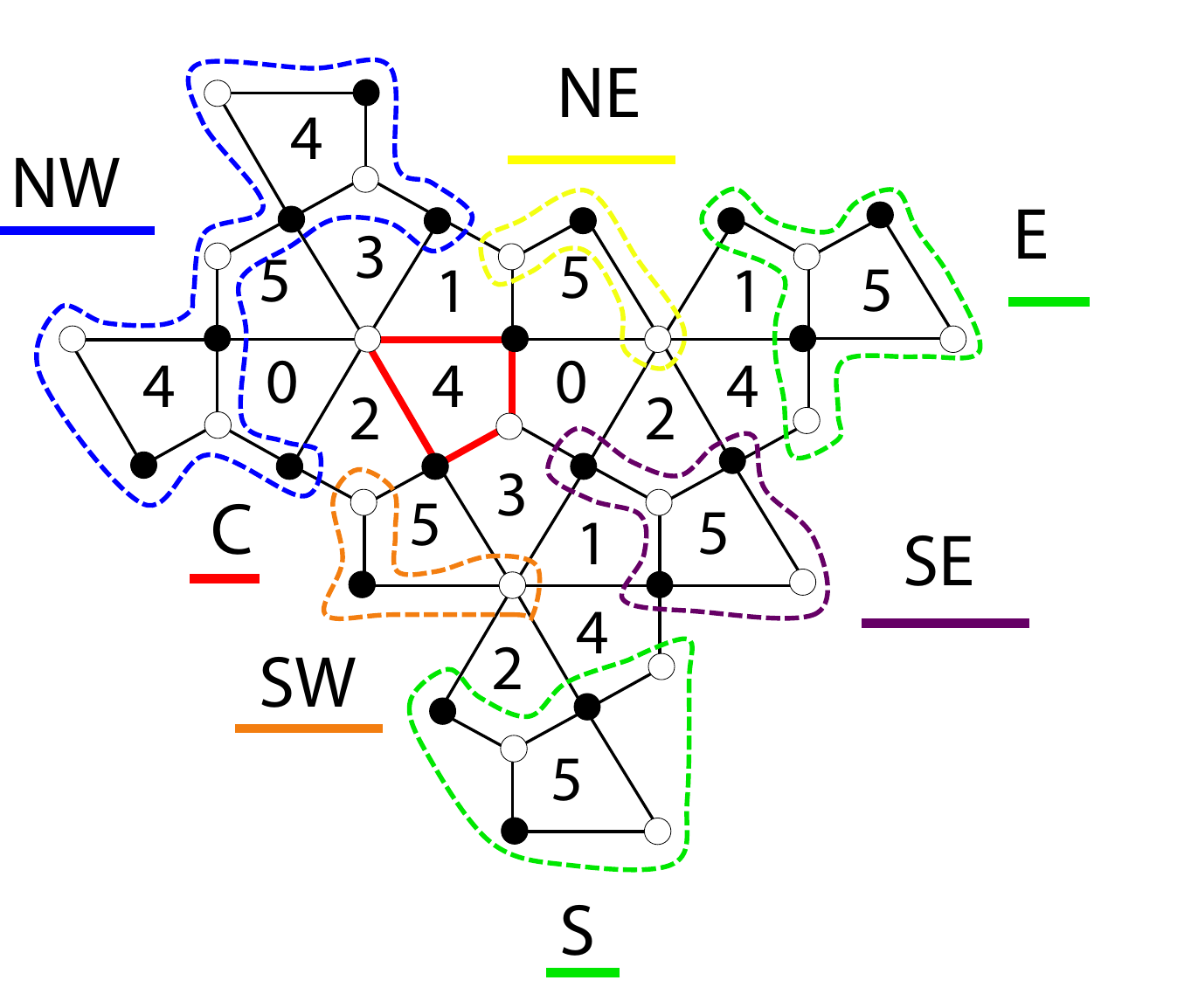}
    \caption{\small The condensation of the castle $\tilde{\gamma}^{-2}_{-1}$.}
    \label{fig:halfDragon}
\end{figure}

Now suppose $j < -2$. Select a basepoint $p$ for $\tilde{\gamma}^j_i$ and let $p_0$ be the endpoint of the path $\tilde{P}_1^1 \tilde{P}_3^1$ beginning at $p$, let $p_1$ be the endpoint of the path $\tilde{P}_3^1$, and let $p_2$ be the endpoint of the path $\tilde{P}_1^1$. We want to decompose the graph $\tilde{\gamma}^{j}_i$. To do so, we simply let $\iota$ act on all of the graphs/vertex sets used in the construction of the decompositions of $\gamma^{-j-1}_{-i}$ in Lemma \ref{lem:genweight}. It is straightforward to check that these decompositions satisfy the conditions of the condensation theorem, and hence the desired weight relation follows. 

\end{proof}

\end{subsubsection}

\begin{subsubsection}{Covering Monomial Relations}

\begin{lemma}
\label{lem: half int cov mon}
For $0 \geq i > j$: 
\begin{equation*}
\begin{split}
m (\tilde{\gamma}^j_ {i}) m (\sigma \tilde{\gamma}^{j + 1} _ {i+1}) &= (x_0x_1x_2x_3x_4x_5) m (\sigma\tilde{\gamma}^{j + 1} _ {i}) m (\tilde{\gamma}^{j}_ {i + 1}) \\ &= (x_0x_1x_2x_3x_4x_5)m (\tilde{\gamma}^{j + 1} _ {i})m (\sigma\tilde{\gamma}^{j}_ {i + 1}) 
\end{split}
\end{equation*}

\end{lemma}

\begin{proof}

The case $i = -1$, $j = -2$ is easily checked by hand. For $j < -2$, as in the proof of Lemma \ref{lem:int cov mon}, the condensations in Lemma \ref{lem:genweight1} naturally give rise to the desired overlayings. Alternatively, apply $\iota$ to those overlappings occuring in Lemma \ref{lem:int cov mon}.

\end{proof}

\end{subsubsection}

\begin{subsubsection}{Proof of the Southwest Cone Theorem and Main Theorem}

Now we can describe the correspondence between SW Aztec castles and alcoves of the SW cone, region VII. 

\begin{theorem}[Main Result 2]
\label{thm:halfintcone}

(1) For $0 \geq i > j$, let $\{i,j\}$ label an odd alcove in the SW cone according to Figure \ref{fig:oddeven}.  Let $\tau_{a_1}\tau_{a_2}\cdots \tau_{a_{-2i-2j-1}}$ denote 
the canonical path from the origin to this alcove. Let 
$y_{i,j}$ and $y_{i,j}'$ denote the penultimate and final cluster variables produced by $\tau_{a_{-2i-2j-1}}$. 

Then $y_{i,j} = 
\begin{cases}  
c(\sigma \tilde{\gamma}_i^j) &\mathrm{~if~} |i| \equiv 0 \mod 3\\ c(\tilde{ \gamma}_i^j) &\mathrm{~otherwise}\end{cases}$, and 

$y_{i,j}' = 
\begin{cases} 
c(\tilde{\gamma}_i^j) &\mathrm{~if~} |i| \equiv 0 \mod 3\\ c(\sigma \tilde{\gamma}_i^j) &\mathrm{~otherwise}\end{cases}$.

(2) For $-1 \geq i > j$, let $( i,j+1 )$ label an even alcove in the NE cone. Let $\tau_{a_1}\tau_{a_2}\cdots \tau_{a_{-2i- 2j}}$ denote 
the canonical path from the origin to this alcove.  This time, let  
$z_{i,j+1}$ and $z_{i,j+1}'$ denote the penultimate and final cluster variables produced by $\tau_{a_{-2i - 2j}}$, respectively.

Then $z_{i,j+1} = 
\begin{cases} c(\tilde{\gamma}_i^{j+1}) &\mathrm{~if~} |i| \equiv 2 \mod 3, \\ c(\sigma\tilde{\gamma}_i^{j+1}) &\mathrm{~otherwise}\end{cases}$, and

$z_{i,j+1}' = 
\begin{cases} c(\sigma \tilde{\gamma}_i^{j+1}) &\mathrm{~if~} |i| \equiv 2 \mod 3, \\ c(\tilde{\gamma}_i^{j+1}) &\mathrm{~otherwise}\end{cases}$.

\end{theorem}

This theorem follows using Lemmas \ref{lem:genweight1} and \ref{lem: half int cov mon} along with applying the same techniques as in Section \ref{sec:proofintcone}.

With the above result, Lemma \ref{cor:zhangperm}, implies that our main theorem, Theorem \ref{thm:main}, is proven by applying the appropriate permutation $\alpha$ for alcoves in regions II - VI, VIII-XII of the Coxeter lattice. Such $\alpha$ are given in Table \ref{tab:perms2}. 

\end{subsubsection}

\end{section}

\begin{section}{Further Questions}
\label{sec: conclusion}

Associating subgraphs to the cluster variables produced by $\tau$-mutation sequences on the dP3 quiver presents a foundation from which to explore several promising directions. We now summarize some possibilities for future investigation.

\begin{problem}To any matching of a planar bipartite graph $G$ it is possible to associate a corresponding \textbf{height function} on the set of matchings of $G$, see for example \cite{dimermodel}, or \cite{CY}. In \cite{JMZpaper}, it is proven that if we enrich our quiver with principal coefficients, the cluster variables produced by the sequence $123\cdots$ are still given by the weighted matchings of Aztec dragons where the coefficients correspond to the height of each perfect matching. In particular, the second author conjectures that Aztec castles equipped with height functions also account for principal coefficients.
\end{problem}

 \begin{problem} One can examine more closely the graph-theoretic properties of Aztec castles. For example, what does the distribution of random perfect matchings look like? Do we have some ``Arctic circle'' behavior as shown for the classical Aztec diamonds in \cite{propp}? Also, in \cite{CY} a domino-shuffle is outlined on the dP$3$ brane tiling; what happens when this shuffle is applied to perfect matchings of Aztec castles? Is there a simple way to see that the number of these matchings is a power of $2$? Even if the shuffling does not give us perfect matchings of other Aztec castles, it could give a new class of subgraphs of the brane tiling whose number of perfect matchings are a power of $2$. \end{problem}

\begin{problem} Now that we have analyzed $\tau$-mutation sequences, it is natural to try to extend the work to further mutation sequences on the dP$3$ quiver, and in particular to toric sequences (mutating only at vertices with $2$ incoming and $2$ outgoing arrows) that are not $\tau$-mutation sequences. While the factorization phenomenon no longer occurs in these cases, a search for subgraphs with appropriate weighted enumerations of perfect matchings is still of interest. A further natural extension would be interpret the other three toric phases of dP$3$ (see Figure 27 of \cite{franco_eager}) in terms of perfect matchings of graphs in the corresponding brane tiling. \end{problem}

\begin{problem} Alexander Garver suggests looking at certain products of cluster variables produced by $\tau-$mutation sequences, as they encode data on pairs of perfect matchings of Aztec castles. Specifically, given two such cluster variables $x_{i}$ and $x_{j}$, if we may write $x_{i}x_{j} = \displaystyle\sum_{k}c_kx_k$, where the $x_k$'s are cluster variables, how do the coefficients $c_k$ encode combinatorial data on superpositions of perfect matchings of the castles associated to  $x_{i}$ and $x_{j}$?
\end{problem}

 \begin{problem}
 \label{prob:ef}
This research project was jumpstarted by applying the method for constructing shadows as outlined in Sections 4 and 7.3 of \cite{franco_eager} by Richard Eager and Sebasti\'{a}n Franco. Their work presents two algorithms for computing subgraphs corresponding to toric mutations. During this investigation, a discrepancy between the graphs obtained by these two methods was observed, for instance in the special case of the $\tau$-sequence $1213\ldots$, but this was soon resolved through a private discussion between the second author and Franco \cite{FM}. In spite of this issue, the method was still instrumental in allowing us to construct the Aztec castles appearing in this paper. This and other observations suggest that there is much work to be done in providing a mathematical basis to the algorithm -- one which could not only shed light on the connections between the combinatorics of cluster algebras and gauge/string theory duality, but also immediately lead to combinatorial interpretations of the type given in this paper for a wide variety of quivers and their associated brane tilings. 
\end{problem}

\end{section}

\section*{Acknowledgments}
\par{
This research was conducted during the 2013 REU in Combinatorics at the University of Minnesota, Twin Cities. The authors would first like to thank Pavlo Pylyavskyy, Joel Lewis, and Dennis Stanton for mentoring the REU alongside the second author. We are grateful that Pavlo Pylyavskyy pointed out the connection to the Coxeter lattice used throughout the paper and that Joel Lewis provided some insightful observations on the factorization phenomenon. The authors would like to thank Thomas McConville for a useful conversation regarding canonical paths, and Sebasti\'
{a}n Franco for correspondence regarding the methods in his work with Richard Eager. 
We also used Mathematica \cite{Mathematica} and the cluster algebras package \cite{Sage2} in Sage \cite{Sage} for numerous computations which were pivotal in finding the patterns leading to our main results.  Thanks also goes to the graduate student mentor Alexander Garver for numerous fruitful conversations and for commenting diligently on drafts of the manuscript, and to Rick Kenyon for a clarifying conversation. Finally, our gratitude extends to two anonymous referees whose insightful comments were instrumental in the construction of the final product. This research was supported by NSF grants DMS-$1067183$ and DMS-$1148634$.

}

\bibliographystyle{alphaurl}

\newcommand{\etalchar}[1]{$^{#1}$}

\end{document}